\documentclass[11pt]{amsart}
\usepackage{amssymb}
\usepackage[usenames, dvipsnames]{color}
\usepackage{amsmath,amscd,amssymb,amsthm,latexsym}
\usepackage{mathrsfs}
\usepackage{tikz}
\usepackage[left=2.3cm,right=2.3cm,top=2.8cm,bottom=2.8cm]{geometry}
\newtheorem{theorem}{Theorem}[section]
\newtheorem{remark}{Remark}[section]
\newtheorem{lemma}{Lemma}[section]
\newtheorem{proposition}{Proposition}[section]
\newtheorem{corollary}{Corollary}[section]
\newtheorem{definition}{Definition}[section]

\usepackage[pagewise]{lineno}
\allowdisplaybreaks
\numberwithin{equation}{section}	

\title[Fractional Kirchhoff Problem]{Nehari manifold approach for fractional Kirchhoff problems with extremal value of the parameter}
\author[Pawan Kumar Mishra]{P. K. Mishra}
\address[Pawan Kumar Mishra]{\newline\indent
	Department of Mathematics
	\newline\indent 
	Indian Institute of Technology Bhilai
	\newline\indent
	492015, Raipur, Chhattisgarh, India}
\email{pawan@iitbhilai.ac.in}
\author[Vinayak Mani Tripathi]{V. M. Tripathi}
\address[Vinayak Mani Tripathi]{\newline\indent
	Department of Mathematics
	\newline\indent 
	Indian Institute of Technology Bhilai
	\newline\indent
	492015, Raipur, Chhattisgarh, India}
\email{vinayakm@iitbhilai.ac.in}
\begin{document}
	\begin{abstract}
		\noindent In this work we study the following  nonlocal problem 
		\begin{equation*}
			\left\{
			\begin{aligned}
				M(\|u\|^2_X)(-\Delta)^s u&= \lambda {f(x)}|u|^{\gamma-2}u+{g(x)}|u|^{p-2}u &&\mbox{in}\ \ \Omega, \\
				u&=0                                   &&\mbox{on}\ \ \mathbb R^N\setminus \Omega,
			\end{aligned}
			\right.
		\end{equation*}
		where  $\Omega\subset \mathbb R^N$ is open and bounded with smooth boundary,  $N>2s, s\in (0, 1), M(t)=a+bt^{\theta-1},\;t\geq0$ with $ \theta>1, a\geq 0$ and $b>0$. The exponents satisfy $1<\gamma<2<{2\theta<p<2^*_{s}=2N/(N-2s)}$ (when $a\neq 0$) and $2<\gamma<2\theta<p<2^*_{s}$ (when $a=0$). The parameter $\lambda$ involved in the problem is real and positive.  The problem under consideration has nonlocal behaviour due to the presence of nonlocal fractional Laplacian operator as well as the nonlocal Kirchhoff term $M(\|u\|^2_X)$, where  $\|u\|^{2}_{X}=\iint_{\mathbb R^{2N}} \frac{|u(x)-u(y)|^2}{\left|x-y\right|^{N+2s}}dxdy$. The weight functions $f, g:\Omega\to \mathbb R$  are continuous, $f$ is positive while $g$ is allowed to change sign. In this paper an extremal value of the parameter, a threshold to apply Nehari manifold method, is characterized variationally for both degenerate and non-degenerate Kirchhoff cases to show an existence of at least two positive solutions  even when $\lambda$ crosses the extremal parameter value by executing fine analysis based on fibering maps and Nehari manifold.
		

		
	\end{abstract} 

\maketitle
	\section{Introduction}
	In the paper, we study the following nonlocal problem
	\begin{equation}\label{p}
		\tag{$P_{\lambda}$}
		\left\{
		\begin{aligned}
			M(\|u\|^2_X)(-\Delta)^s u&= \lambda {f(x)}|u|^{\gamma-2}u+{g(x)}|u|^{p-2}u &&\mbox{in}\ \ \Omega, \\
			u&=0                                   &&\mbox{on}\ \ \mathbb R^N\setminus \Omega,
		\end{aligned}
		\right.
	\end{equation}
	where  $\|u\|^{2}_{X}=\iint_{\mathbb R^{2N}} \frac{|u(x)-u(y)|^2}{\left|x-y\right|^{N+2s}}dxdy$,  $\Omega\subset \mathbb R^N$ is open and bounded with smooth boundary,  $N>2s, s\in (0, 1), M(t)=a+bt^{\theta-1},\;t\geq0$ with $ p/2>\theta>1, a\geq 0$ and $b>0$. The exponents satisfy $1<\gamma<2<{2\theta<p<2^*_{s}}$ (when $a\neq 0$) and $2<\gamma<2\theta<p<2^*_{s}$ (when $a=0$), where  $2^*_{s}=2N/(N-2s)$ is fractional Sobolev critical exponent. Here $(-\Delta)^{s}$ is well know fractional Laplacian operator which is defined, up to a normalization constant, as
	$$(-\Delta)^s \varphi(x)=\int_{\mathbb{R}^N}\frac{2\varphi(x)-\varphi(x+y)-\varphi(x-y)}{|y|^{N+2s}}dy,\quad x\in\mathbb R^N,$$
	for any $\varphi\in C_0^\infty(\Omega)$. The weight functions $f, g:\Omega\to \mathbb R$  are continuous and  satisfy  following assumptions:
	\begin{itemize}
		\item [(F)]  $f\in C(\Omega)\cap L^\infty(\Omega) $ and there exists some $f_0>0$ such that $f(x)>f_0$ for all $x$ near $\partial \Omega$, 
		\item [(G)] $g\in C(\Omega)\cap L^\infty(\Omega)$ with $0\not\equiv g^+=\max\{0, g(x)\}.$ 
	\end{itemize}
	
	The Kirchhoff term is called degenerate if $a=0$ and non-degenerate otherwise. The problem is non-local due to the presence of the operator  formed by the nonlocal Kirchhoff term and fractional Laplacian operator. The nonlocal nature of the operator does not allow us to compare the equation in \eqref{p} pointwise.  The	problem \eqref{p} has a variational structure and the suitable functional space (inspired from \cite{Val1})  to look for solutions can be 
	$$
	X=\big\{u\in H^s(\mathbb R^N):\,\,u=0\mbox{ a.e. in } \mathbb R^N\setminus \Omega\big\},
	$$
	where $H^s(\mathbb R^N)$ is the fractional Sobolev space (see \cite{Hitchhker} for more details). We recall that $X$ is a Hilbert space endowed with the following norm
	\begin{equation*}
		\left\|u\right\|_{X}:=\left(\iint_{\mathbb R^{2N}} \frac{|u(x)-u(y)|^2}{\left|x-y\right|^{N+2s}}dxdy\right)^{1/2}.
	\end{equation*}
	\begin{definition}\label{def}
		A function $u\in X$ is said to be a (weak) solution of \eqref{p} if for every $\phi\in X$ 
		\begin{equation*}
			\begin{aligned}
				(a+b\|u\|^{2(\theta-1)}_{X})&\iint_{\mathbb R^{2N}}  \frac{(u(x)-u(y))(\phi(x)-\phi(y))}{|x-y|^{N+2s}}dxdy\\&-{\lambda}\int_\Omega f u^{\gamma-1} \phi dx -\int_\Omega gu^{p-1}\phi dx=0.
			\end{aligned}
		\end{equation*}
	\end{definition}
	
	To study the problem via variatonal methods, we define the associated  energy functional $\mathcal E_{\lambda}: X\to \mathbb{R}$ as
	\begin{equation*}
		\mathcal E_{\lambda}(u)=\frac{a}{2}\|u\|^2_{X}+\frac{b}{2\theta}\|u\|^{2\theta}_{X}-\frac{\lambda}{\gamma}\int_\Omega f|u|^{\gamma} dx -\frac{1}{p} \int_\Omega g|u|^p dx.
	\end{equation*}

	One can see that $\mathcal{E}_{\lambda}(u)$ is of class $C^1(X)$ and solutions of \eqref{p}, in the light of above Definition \ref{def}, can be viewed as the critical points of $\mathcal{E}_{\lambda}(u)$. The class of problem under consideration has widely studied  in recent past  because of their vast application in many areas of science. The above  class of problems can be seen as stationary state of the following problem	
	\begin{equation*}
		\left\{
		\begin{aligned}
			u_{tt}-M\left(\int_{ \Omega}|\nabla u|^2dx\right)\Delta u&= F(x,u) &&\mbox{in}\ \ \Omega, \\
			u&=0                                   &&\mbox{on}\ \ \partial\Omega,
		\end{aligned}
		\right.
	\end{equation*}
	which was firstly introduced by Kirchhoff  to deal with free transversal oscillations of elastic strings (see \cite{Kirchhoff}). The term $M$ measures the tension in the string caused by any change in the length of the string during vibration and is directly proportional to the Sobolev norm of the displacement of the string. The degenerate case, arises if the initial tension of the string is zero, physically a very realistic model. We refer the interested readers to the servey \cite{PUCCI} for recent advances in nonlocal Kirchhoff problems. 
	
	The class of nonlinearity under consideration in this paper is referred as concave-convex nonlinearity and attracted many researchers to explore after the pioneering work of Ambrosetti et al. in  \cite{Ambrosetti} where, authors have established global multiplicity result.  
	The work on this class of nonlinearity can be found in \cite{Azorero, Barrios,kj, tfw,Wu} and references therein for both local and {nonlocal problems}. In particular, \cite{Barrios} and \cite{Azorero} contain the study of quasilinear and fractional counter part of \cite{Ambrosetti}, respectively. Precisely the problem involving non-local fractional operator of the type
	\[(-\Delta)^s u=\lambda f u^{q-1}+gu^{p-1},\; u>0, \ \ \textrm{in} \ \Omega, \ \ u=0 \ \textrm{on} \  \mathbb R^n\setminus\Omega,\] 
	has been studied in \cite{sarika}( see \cite{tfw, Wu} for $s=1$) for existence of at least two solutions when $1<q<2$ and ${1<p<2^*_s-1}$ via Nehari manifold technique introduced by Pokhozhaev (see \cite{Poho}) and Nehari (see \cite{N1,N2}). The multiplicity results obtained obtained in \cite{sarika} (also in \cite{tfw, Wu}) heavily relies on the fact that the  nonlinearity  exhibit the convex-concave behaviour which is essential for fibering maps to have two critical points.  In this paper, we have shown that the degenerate fractional Kirchhoff operator allows us to break this structure of concave-convex nonlinearity in order to show multiple solutions via Nehari manifold technique.    
	
	Another class of nonlocal problems involves the operator formed with the fusion of Kirchhoff term and fractional Laplacian. In \cite{Valdinoci} (see the Appendix) authors have given a motivation for this class of problem by modelling a fractional Kirchhoff problem coming out of  string vibrations. These Kirchhoff counterparts  have also occupied the enough space in the  literature due to its vast applications in applied science, see for reference \cite{Autuori,Binlin,NA2,Valdinoci,mana,saldi,pucci,Molica} and references therein . We also refer readers to  \cite{chen,mana} for non-degenerate structure of Kirchhoff term where results have been established  with a control on the parameter $\lambda<\Lambda$ for some small $\Lambda>0$ via Nehari manifold minimization argument. The control on the parameter $\lambda$ is natural as for  $\lambda\geq \Lambda$ it is not obvious to show that the constrained minimizers obtained in different decompositions of Nehari set are the critical points of the associated energy functional of the problem because $\mathcal N_\lambda$ is no longer a manifold.

	In this paper, inspired from the work of \cite{Il'yasov,Silva}, we have included this delicate case in our study when $\lambda \geq \lambda^*_{a,b}$ where $\lambda^*_{a,b}$ is termed as extremal parameter introduced in \cite{IL2} while studying generalized Rayleigh quotient. The value $\lambda^*_{a,b}$  is extremal in the sense that the Nehari set will no longer remain a manifold if the parameter $\lambda$ crosses this threshold $\lambda^*_{a,b}$. By overcoming the above technical difficulty, for $\epsilon>0$ sufficiently small, we have shown existence of at least two positive solutions by constrained minimization on Nehari manifold based on fibering analysis for $\lambda\in (0, \lambda^*_{a,b}+\epsilon)$. 
	
	The work of this paper is divided in the following order.  In section [2],  we have  provided  framework of Nehari setup and section [3] includes some technical results together with our main result of the paper. In section [4], [5] and [6], we have shown the existence of minimizers for the energy functional when $\lambda<\lambda^*_{a,b},\lambda=\lambda^*_{a,b}$ and $\lambda>\lambda^*_{a,b}$ respectively. We have attempted to address both the degenerate and non-degenerate Kirchhoff cases collectively and distinguished whenever it's required.
	
	\section{Nehari manifold structure}
	The main objective of this paper is to look for critical points of $\mathcal E_\lambda$ over $X$.	The critical points of   $\mathcal E_\lambda$ can be obtained via minimization of $\mathcal E_\lambda$. Since $\mathcal E_\lambda$, fails to be bounded from below in $X$ and therefore, we constrained the minimization problem  in a proper subset of $X$, namely Nehari manifold (a natural constraint as it contains all the critical points of $\mathcal E_\lambda$) defined as
	\begin{equation*}
		\mathcal{N}_{\lambda}=\left\{u\in  X \setminus\{0\}: D_{u}\mathcal E_\lambda(u)(u)=0\right\}.
	\end{equation*}
	It can be observed that $\mathcal{N}_{\lambda}$ is a disjoint union of the following sets,
	\begin{align*}
		\mathcal{N}_{\lambda}^+=\{
		u\in  \mathcal N_{\lambda}:&D_{uu}\mathcal E_\lambda(u)(u, u)>0\},
		\mathcal{N}_{\lambda}^-=\{
		u\in  \mathcal N_{\lambda}:D_{uu}\mathcal E_\lambda(u)(u, u)<0\},	\\
		\mathcal{N}_{\lambda}^0=\{
		u\in  \mathcal N_{\lambda}:&D_{uu}\mathcal E_\lambda(u)(u, u)=0\}.	
	\end{align*}
	Consider the fibering functions $\psi_{\lambda, u} :[0, \infty)\to \mathbb R$ for every $u\in X$, defined as $\psi_{\lambda,u}(t)=\mathcal E_{\lambda}(tu)$, that is,
	\[
	\psi_{\lambda,u}(t)=\frac{at^2}{2}\|u\|^2_{X}+\frac{b t^{2\theta}}{2\theta}\|u\|^{2\theta}_{X}-\frac{\lambda t^{\gamma}}{\gamma}\int_\Omega f |u|^{\gamma} dx -\frac{t^{p}}{p}  \int_\Omega g|u|^p dx.
	\]
	Note that $u\in \mathcal{N}_{\lambda}$ if and only if $\psi'_{\lambda,u}(1)=0$ and, in general,  $tu\in \mathcal N_{\lambda}$ if and only if $\psi'_{\lambda, u}(t)=0$. Therefore, we can read  Nehari submanifolds in terms of $t=1$ being local maxima or local minima or saddle point of $\psi_{\lambda, u}(t)$ as follows
	\begin{align*}
		\mathcal N_{\lambda}^+&=\{u\in \mathcal N_{\lambda}: \psi_{\lambda,u}''(1)>0\},\;
		\mathcal N_{\lambda }^-=\{u\in \mathcal N_{\lambda}: \psi_{\lambda,u}''(1)<0\},\\
		\mathcal N_{\lambda}^0&=\{u\in \mathcal N_{\lambda}: \psi_{\lambda,u}''(1)=0\}.
	\end{align*}
	The presence of sign changing weights as well as the nature of Kirchhoff term can govern the behaviour of fibering functions and therefore we split the discussion based on the sign of the integral involving sign changing weight $g(x)$ and the nature of Kirchhoff term. For, we define the following

	\begin{align*}
		\mathcal{C}^+&=\{u\in X: \int_\Omega g|u|^p dx>0\} \;\;\textrm{and}\;\;
		\mathcal{C}^-=\{u\in X:\int_\Omega g|u|^p dx\leq 0\}.
	\end{align*} 
	
	In the next section we give a characterization of extremal parameter $\lambda^*_{a,b}$ for the both degenerate and non-degenerate Kirchhoff term.   
	
	\section{Characterization of extremal parameter}
	For $u \in \mathcal{C}^+$,	by	solving  following system of two nonlinear equations in variables $t$ and $\lambda$
	\begin{equation}\label{one}	\left\{
		\begin{aligned}a\|tu\|^2_{X}+b\|tu\|^{2\theta}_{X}-\lambda \int_\Omega f|tu|^{\gamma}dx-t^{p}\int_\Omega g(x)|u|^pdx&=0,\\
			2a\|tu\|^2_{X}+2b\theta\|tu\|^{2\theta}_{X}-\lambda\gamma\int_\Omega f|tu|^{\gamma}dx-p \ t^{p} \int_\Omega g(x)|u|^pdx&=0,\end{aligned}
		\right.
	\end{equation}
	we can find $(t_{a,b}(u),\lambda_{a,b}(u))$ uniquely.
	In fact, 
	by eliminating $\lambda$ from above equations we can see $t_{a,b}(u)$ as unique zero of the  following scalar function
	\[m_u(t)=a(2-\gamma)t^2\|u\|^2_{X}+b(2\theta-\gamma)t^{2\theta}\|u\|^{2\theta}_{X}-{(p-\gamma)t^{p}}\int_\Omega g(x)|u|^pdx,\] and  hence we have unique $\lambda_{a,b}(u)$, given as 
	\[\lambda_{a,b}(u)=\frac{a(p-2)\|t_{a,b}(u)u\|^2_{X}+2b(p-2\theta)\|t_{a,b}(u)u\|^{2\theta}_{X}}{(p-\gamma)\int_{ \Omega}f|t_{a,b}(u)u|^\gamma dx}.\]
	In the degenerate Kirchhoff case, we can get an explicit expression for the solution of the above system \eqref{one} as 
	\[t_{0,1}(u):=t(u)=\left(\frac{(2\theta-\gamma)\|u\|^{2\theta}_{X}}{{(p-\gamma)}\int_\Omega g(x)|u|^p dx}\right)^{\frac{1}{p-2\theta}} \]
	and
	\begin{align*}
		\lambda_{0,1}(u):=	\lambda(u)
		&=\left(\frac{p-2\theta}{p-\gamma}\frac{\|u\|^{2\theta}_{X}}{\int_{\Omega}f |u|^\gamma dx}\right)\left(\frac{2\theta-\gamma}{p-\gamma}\frac{\|u\|^{2\theta}_{X}}{\int_\Omega g(x)|u|^pdx}\right)^{\frac{2\theta-\gamma}{p-2\theta}}.
	\end{align*}
	For $a\geq 0, b>0$	define the extremal value  
	\begin{equation}\label{lmini}
		\lambda^*_{a, b}=\inf_{X \cap C^+}\lambda_{a,b}(u)
	\end{equation} In particular, when $a=0, b=1$, we denote $\lambda^*=\lambda^*_{0, 1}$.
	
	The following proposition is direct consequence of compact  fractional Sobolev embeddings.
	
	\begin{proposition}\label{luch}
		The function $\lambda_{a,b}(u)$ is continuous, 0-homogeneous and unbounded from above. Moreover, $\lambda^*_{a,b}>0$  and there exists $u\in \mathcal{C}^+$ such that $\lambda^*_{a,b}=\lambda_{a,b}(u)$.
		
	\end{proposition}
	\begin{remark}\label{lwsc}
		The function $\lambda_{a,b}(u)$ is weakly lower semicontinuous also. {For the degenerate case, it follows directly form weak lower semi-continuity of the norm. For non-degerate case,  let $u_n \rightharpoonup u$ in $\mathcal{C}^+$. If  $\|u_n\|_{X}\rightarrow \|u\|_{X},$ then from uniqueness of the solution of the system in \eqref{one}, we get $t(u_n)\rightarrow t(u)$ and $\lambda(u_n)\rightarrow\lambda(u)$. If not, then we have $\|u\|_X<\liminf_{n\rightarrow\infty}\|u_n\|$. Therefore,
			$0=\psi'_{\lambda_{a,b},u}(t(u))<\liminf_{n\rightarrow\infty}\psi'_{\lambda_{a,b},u_n}(t(u)),$ i.e., $\psi'_{\lambda,u_n}(t(u))>0$ for large value of $n$. Thus graph of $\psi_{\lambda_{a,b},u_n}(t)>0$ will be of the type fig 2(b) for large values of $n$. Equivalently,  $\lambda_{a,b}(u)<\lambda_{a,b}(u_n)$ for large $n$ or $\lambda_{a,b}(u)\leq\liminf_{n\to\infty}\lambda_{a,b}(u_n)$.} 
	\end{remark}

	In the following Proposition, we show that Nehari decompositions are non-empty.
	\begin{proposition}{\label{p,2.2}}
		For a given $u\in X$ and $a>0, b>0$ with $1<\gamma<2$ or $a=0, b>0$ with $2<\gamma<2\theta$ there are the following two cases:
		
		\begin{enumerate}
			\item [(i)]If $u\in \mathcal \mathcal{C}^-$, then for  any $\lambda>0$, there exists a unique $t_{\lambda}^+(u)>0$ such that  $t_{\lambda}^+u \in \mathcal N_{\lambda}^+.$ 
			\item [(ii)]If $u\in \mathcal{C}^+$, then we have following three situations.
			\begin{enumerate}
				\item [(a)]  $0<\lambda<\lambda_{a,b} (u)$, there exist unique $0<t_{\lambda}^+(u)<t_{\lambda}^-(u)$ such that $t_{\lambda}^+ u\in \mathcal N_{\lambda}^+ $ and $t_{\lambda}^-u\in \mathcal N_{\lambda}^- $. 
				\item [(b)]  $\lambda=\lambda_{a,b}(u)$, then there exists unique $t^0_{\lambda}(u)>0$ such that $t^0_{\lambda}u\in \mathcal N_{\lambda}^0$.
				\item [(c)] $\lambda>{\lambda_{a,b}}(u)$, then $tu\not \in \mathcal N_{\lambda}$ for any $t>0$. In particular, $u\not\in \mathcal N_{\lambda}.$
			\end{enumerate}
		\end{enumerate}
	\end{proposition}
	
	\begin{proof}
		Proof of this proposition is divided for degenerate and non-degenerate cases respectively as follows.\\
		\textbf{Case 1:}(Degenrate Kirchhoff case) To see the graph of fibering map $\psi_{\lambda,u}$ define a map  $\Phi_{u}:\mathbb{R}^{+}\rightarrow\mathbb{R}$ such that
		\[\Phi_{u}(t)=t^{2\theta-\gamma}\|u\|^{2\theta}_{X}-t^{p-\gamma}\int_{ \Omega}g|u|^pdx,\]
		which has a unique critical point when $u\in\mathcal{C}^{+}$ and can be obtained by solving	 the following equation 
		\[\Phi'_{u}(t)=(2\theta-\gamma)t^{2\theta-\gamma-1}\|u\|^{2\theta}_{X}-(p-\gamma)t^{p-\gamma-1}\int_{ \Omega}g|u|^pdx.\] 
		Moreover, as $\lim_{t\rightarrow 0^+}\Phi_{u}(t)=0$ and $\lim_{t\rightarrow \infty}\Phi_{u}(t)=-\infty$ the graph of $\Phi_u$ be as shown in the fig 1(b). When  $u\in \mathcal{C}^{-}$,  $\Phi'_{u}(t)>0$ for all $t\geq0$ and hence no critical point for $\Phi_u(t)$ leading to the graph as shown in fig 1(a).  
		
		\begin{figure}
			
			\flushleft
			\begin{tikzpicture}[scale=.45]
				\draw[thick, ->] (-1, 0) -- (7, 0);
				\draw[thick, ->] (0, -1) -- (0, 5);
				\draw[thick] (0,0) .. controls (3.8,2) and (3.8,2) .. (4.6,5);
				\draw (7,0) node[below]{$t$};
				\draw (-1.2,5.5) node[below]{$\Phi_u(t)$};
				\draw (11,-7) node[below]{fig 1: Behaviour of $\Phi_u(t)$ };
				\draw (2,-3.1) node[below]{fig 1(a)\;$u\in \mathcal{C}^-$};
				\hspace*{6cm}
				\draw[thick, ->] (-1, 0) -- (7, 0);
				\draw[thick, ->] (0, -1) -- (0, 5);
				\draw[thick] (0,0) .. controls (0,0.2) and (0,0.2) .. (1,1.2);
				\draw[thick] (1, 1.2) .. controls (3.3,3)  and  (4.3, 0) ..(4.33,0);
				\draw[thick] (4.33, 0) .. controls (4.4,0) and  (4.2, 0) ..(5.8,-3.2);
				\draw (7,0) node[below]{$t$};
				\draw (-1.2,5.5) node[below]{$\Phi_u(t)$};
				\draw (2,-3.1) node[below]{fig 1(b)\;$u\in  \mathcal{C}^+$};
			\end{tikzpicture}
			
		\end{figure}

		\textbf{Case 2:}(Non-degenrate Kirchhoff case) For a given $u\in X\setminus 0$, define a  map  $\Phi_{u}:\mathbb{R}^{+}\rightarrow\mathbb{R}$ such that
		\[\Phi_{u}(t)=at^{2-\gamma}\|u\|^2_{X}+bt^{2\theta-\gamma}\|u\|^{2\theta}_{X}-t^{p-\gamma}\int_{ \Omega}g|u|^pdx.\]
		Note that $tu\in \mathcal N_\lambda$ (or $t$ is a critical point of $\psi_{\lambda, u}$) if and only if the following equation has a scalar solution in $t$
		\begin{equation}\label{sceq}
			\Phi_{u}(t)=\lambda\int_\Omega f|u|^\gamma dx\end{equation}
		In order to look for solution of above scalar equation \eqref{sceq}, we analyze the behaviour of $\Phi_u(t)$. For,  \[\Phi'_{u}(t)=a(2-\gamma)t^{1-\gamma}\|u\|^2_{X}+b(2\theta-\gamma)t^{2\theta-\gamma-1}\|u\|^{2\theta}_{X}-(p-\gamma)t^{p-\gamma-1}\int_{ \Omega}g|u|^pdx.\]
		When $u\in \mathcal{C}^-$, $\Phi'_{u}(t)>0$ for all $t\geq0$ and hence \eqref{sceq} has unique solution for all values of $\lambda>0$. Consequently, unique projection of $u$, namely $t_\lambda^+ u$ in $\mathcal N_\lambda^+$ (as $\psi_{\lambda, u}^{\prime\prime}(t)=t^{1+\gamma} \Phi'_u(t)$) In order to see the behaviour of $\Phi_u(t)$ when $u\in \mathcal{C}^+$,
		we can rewrite $\Phi'_{u}(t)=t^{1-\gamma}\mathcal{H}_{u}(t)$, where
		\[\mathcal{H}_u(t)=a (2-\gamma)\|u\|^{2}_{X}+b(2\theta-\gamma)t^{2\theta-2}\|u\|^{2\theta}_X-(p-\gamma)t^{p-2}\int_{\Omega}g|u|^pdx.\] 
		Then
		\[\mathcal{H}'_{u}(t)=b(2\theta-2){(2\theta-\gamma)}t^{2\theta-1}\|u\|^{2\theta-2}_{X}-(p-2)(p-\gamma)t^{p-1}\int_{ \Omega}g|u|^pdx.\]
		Thus one can observe that there exist a unique critical point $t^*>0$ such that $\mathcal{H}'_{u}(t^*)=0$, where
		\[t^*=\displaystyle{\left(\frac{b(2\theta-2){(2\theta-\gamma)}\|u\|^{2\theta}_{X}}{(p-2)(p-\gamma)\int_{ \Omega}g|u|^pdx}\right)^\frac{1}{p-2\theta}}.\] When $u\in\mathcal{C}^{+}$, $\mathcal{H}'_{u}(t)>0$ as $t\rightarrow 0^+$ and $\mathcal{H}'_{u}(t)\rightarrow-\infty$ as $t\rightarrow \infty$. { As $\mathcal{H}_{u}(t)>0$, as $t\rightarrow 0^+$ and   $\mathcal{H}_{u}(t)\rightarrow-\infty$ as $t\rightarrow \infty$, there exists unique $t_*>t^*>0$ such that $\mathcal{H}_u(t_*)=0$.} Therefore $\Phi_u(t)$ has global maximum at unique point $t=t_*$. Consequently, equation \eqref{sceq} has exactly two solutions for suitable control over $\lambda$. In otherwords, there are unique projections of $u$, namely $t_\lambda^+u$ and $t_\lambda^- u$ such that $t_\lambda^+u\in \mathcal N_\lambda^+$ and $t_\lambda^- u\in \mathcal N_\lambda^-$ for $\lambda<\lambda_{a, b}(u)$. 
		\begin{figure}
			
			\flushleft
			\begin{tikzpicture}[scale=.45]
				\draw[thick, ->] (-1, 0) -- (7, 0);
				\draw[thick, ->] (0, -1) -- (0, 5);
				\draw[thick] (0,0) .. controls (0.7,-0.1) and (0.7,-0.1) .. (0.8,-0.2);
				\draw[thick] (0.8,-0.2) .. controls (3.8,-4)  and  (4, 2) ..(5,4);
				\draw[dotted] (2.42,0) .. controls (2.42,-1.5) and  (2.42,-1.5) ..(2.42,-1.5);
				\draw (2,1.1) node[below]{$t^{+}_{\lambda}(u)$};
				\draw (7,0) node[below]{$t$};
				\draw (-1.2,5.5) node[below]{$\psi_{\lambda,u}$};
				\draw (2,-3.1) node[below]{fig 2(a)\;$u\in \mathcal{C}^-, \;\lambda>0$};
				\hspace*{6cm}
				\draw[thick, ->] (-1, 0) -- (7, 0);
				\draw[thick, ->] (0, -1) -- (0, 5);
				\draw[thick] (0,0) .. controls (0.7,-0.1) and (0.7,-0.1) .. (0.8,-0.2);
				\draw[thick] (0.8, -0.2) .. controls (2.8,-3)  and  (3, 0) ..(3.36,0);
				\draw[thick] (3.3, 0) .. controls (4.6,3) and  (6, 0) ..(6.6,-2);
				\draw[dotted] (2.25,0) .. controls (2.25,-1.5) and  (2.25,-1.5) ..(2.25,-1.5);
				\draw[dotted] (4.5,0) .. controls (4.5,1.35) and  (4.5,1.35) ..(4.5,1.35);
				\draw (4.5,-0.2) node[below]{$t^{-}_{\lambda}(u)$};
				\draw (2,1.1) node[below]{$t^{+}_{\lambda}(u)$};
				\draw (7,0) node[below]{$t$};
				\draw (-1.2,5.5) node[below]{$\psi_{\lambda,u}$};
				\draw (2,-3.1) node[below]{fig 2(b)\;$u\in \mathcal{C}^+, \;\lambda<{\lambda_{a,b}}(u)$};
			\end{tikzpicture}
			\flushleft
			\begin{tikzpicture}[scale=.45]
				\draw[thick, ->] (-1, 0) -- (7, 0);
				\draw[thick, ->] (0, -4) -- (0, 2);
				\draw[thick] (0,0) .. controls (0.6,-0.1) and (0.6,-0.1) .. (0.7,-0.2);
				\draw[thick] (0.7, -0.2) .. controls (2,-2) and  (2, -2) .. (3,-2);
				\draw[thick] (3, -2) .. controls (4,-2) and  (4, -2) ..(5.6,-4);
				\draw[dotted] (2.8,0) .. controls (2.8,-1.9) and  (2.8,-1.9) ..(2.8,-1.9);
				\draw (2.7,1.1) node[below]{$t^{0}_{\lambda}(u)$};
				\draw (2.5,-5.1) node[below]{fig 2(c)\;$u\in \mathcal{C}^+, \;\lambda={\lambda_{a,b}}(u)$};
				\draw (7,0) node[below]{$t$};
				\draw (-1.2,2.5) node[below]{$\psi_{\lambda,u}$};
				\draw (11,-7) node[below]{fig. 2: Behaviour of $\psi_{\lambda,u}(t)$};
				\hspace*{6cm}
				\draw[thick, ->] (-1, 0) -- (7, 0);
				\draw[thick, ->] (0, -4) -- (0, 2);
				\draw[thick] (0,0) .. controls (0.7,-0.1) and (0.7,-0.1) .. (0.74,-0.1);
				\draw[thick] (0.7, -0.1) .. controls (3.7,-1) and  (3.7, -1) ..(6,-4);
				\draw (2,-5.1) node[below]{fig 2(d)\;$u\in \mathcal{C}^+,\;\lambda>{\lambda_{a,b}}(u)$};
				\draw (7,0) node[below]{$t$};
				\draw (-1.2,2.5) node[below]{$\psi_{\lambda,u}$};
			\end{tikzpicture}
		\end{figure}
		
		{The result of this proposition is depicted in the figure fig 2 above. Observe that fig 2$(a)$ corresponds to the part $(i)$  and fig 2$(b)-2(d)$ are related with parts $(ii)(a)-(c)$ respectively.} 
	\end{proof}
	
	\begin{remark}It is clear in light of the above results, that  $\mathcal{N}_{\lambda}^-\neq\emptyset$ and  $\mathcal{N}_{\lambda}^+\neq\emptyset$ for all $\lambda>0$. Moreover for $0<\lambda<\lambda^*_{a,b}$,
		$\mathcal{N}_{\lambda}^0=\emptyset$. By setting, $N(u)=D_{u}\mathcal{E}_{\lambda}(u)u$, we have $D_{u}N(u)(u)=D_{uu}\mathcal{E}_{\lambda}(u)(u,u)+D_{u}\mathcal{E}_{\lambda}(u)u=D_{uu}\mathcal{E}_{\lambda}(u)(u,u)\neq0$ as $u\in \mathcal{N}_{\lambda}$ and $\mathcal{N}_{\lambda}^0=\emptyset$ for $\lambda<\lambda^*_{a,b}$. Therefore from implicit function theorem one can view $\mathcal{N}_{\lambda}$ as $C^1$ manifold of codimension one when $\lambda<\lambda^*_{a,b}$. But once $\mathcal{N}^{0}_{\lambda}\neq \emptyset$, we can have a $u\in \mathcal{N}_{\lambda}$ such that $D_{uu}\mathcal{E}_{\lambda}(u)(u,u)=0$,  in that case the set $\mathcal{N}_{\lambda}$ will be no longer a manifold. In coming section we will see that Nehari decomposition $\mathcal{N}^0_{\lambda}\neq \emptyset$ when $\lambda$ crosses the extremal value.
	\end{remark}
	
	The main results of this paper are stated in the form of the following theorems.
	\begin{theorem}\label{T1}
		Let $g<0$ near $\partial\Omega$, $M$ be defined as $M(t)=a+bt^{\theta-1}$ with $\theta>1$ satisfying $2\theta<p<2^*_s$ and $a> 0$ then for $\gamma\in(1, 2)$ there exist at least two positive solutions for the problem \eqref{p}  for $\lambda\in(0,\lambda^*_{a,b}+\epsilon)$, where $\epsilon>0$ is sufficiently small.	\end{theorem}
	In the sequal for the degenerate Kirchhoff case we established the multiplicity results even when the non-linearity looses the concave-convex behaviour with no control on the sign of $g$ near boundary. Precisely, we have the following result.
	\begin{theorem}\label{T2}
		Let $M(t)=t^{\theta-1}, \theta>1$  satisfying $2\theta<p<2^*_s$, then for $\gamma\in(2,2\theta)$ there exist two positive solutions for \eqref{p} when $\lambda\in(0,\lambda^*+\bar{\epsilon})$, where  $\bar\epsilon>0$ is sufficiently small.
	\end{theorem} 
	\begin{remark} The results obtained in this paper contributes to the literature in the following way. We have  complimented the work of \cite{sarika,Silva} for nonlocal fractional Kirchhoff problem and the work of \cite{sarika} has further extended for the parameter lying beyond the extremal parameter value.  Even when the nonlinearity lacks the sublinear-superlinear combination, we have shown multiplicity results for degenerate Kirchhoff case.
	\end{remark}
	Observe that, from Proposition \ref{p,2.2}, $u$ may has its projection in $\mathcal N_\lambda^-$ only if $u\in \mathcal C^+$ but $\mathcal N_\lambda^+$ can have projections of $u$ irrespective of $u$ lying in $\mathcal C^+$ or not. We require this distinction on later stages while studying the minimization problems. Therefore,  for each $\lambda>0$,  we define the following sets noting this distinction
	\begin{equation}\label{capnl}
		\begin{aligned}
			\mathcal{\hat{N}}_{\lambda}&=\left\{u\in X\setminus\{0\}: u\in \mathcal{C}^+  \ \textrm{and} \ \lambda<\lambda_{a,b}(u) \right\},\\
			\hat {\mathcal{N}}^{+}_{\lambda}&=\left\{u\in X\setminus\{0\}:u\in \mathcal{C}^- \right\}.
		\end{aligned}
	\end{equation}
	
	\begin{remark} \label{remc}One can observe that when $u\in \mathcal{\hat{N}}_{\lambda}\cup \hat {\mathcal{N}}^{+}_{\lambda}$ then $tu\in \mathcal{\hat{N}}_{\lambda}\cup \hat {\mathcal{N}}^{+}_{\lambda}$ as $\lambda_{a,b}(u)$ is $0-$ homogeneous and thus we can say	that the set $\mathcal{\hat{N}}_{\lambda}\cup \hat {\mathcal{N}}^{+}_{\lambda}$ represents a cone generated by $\mathcal{N}_{\lambda}^+\cup \mathcal{N}_{{\lambda}}^-$, that is, 
		\[\mathcal{\hat{N}}_{\lambda}\cup \hat {\mathcal{N}}^{+}_{\lambda}= \{tu: t>0, u\in \mathcal{N}_{\lambda}^+\cup \mathcal{N}_{\lambda}^- \}.\] 
	\end{remark}

	To prove our main result, for all $\lambda>0$, let us consider the following minimization problems
	\[\hat{\mathcal{J}}^{-}_{\lambda}=\textrm{inf}\{\mathcal{J}^{-}_{\lambda}(u):u\in  \mathcal{N}^{-}_{\lambda}\}   \ \ \textrm{and} \ \  
	\hat{\mathcal{J}}^{+}_{\lambda}=\textrm{inf}\{\mathcal{J}^{+}_{\lambda}(u):u\in  \mathcal{N}^{+}_{\lambda}\}, 
	\]
	where functions $\mathcal{J}^{-}_{\lambda}: \hat{\mathcal{N}}_{\lambda}\rightarrow \mathbb{R}$ and  $\mathcal{J}^{+}_{\lambda}: \mathcal{\hat{N}}_{\lambda} \cup \hat {\mathcal{N}}^{+}_{\lambda} \rightarrow \mathbb{R}$ are defined as follows
	\[\mathcal{J}^{-}_{\lambda}(u)= \mathcal E_{\lambda} (t^{-}_\lambda(u)u) \ \  \textrm{and} \ \ \ \mathcal{J}^{+}_{\lambda}(u)= \mathcal E_{\lambda} (t^{+}_\lambda(u)u).\] 
	We have the following observations about ${\mathcal J}^{\pm}_{\lambda}$ as $1<\gamma<p$ and $2<2\theta<p$ .
	\begin{proposition}
		The  functional ${\mathcal J}^{\pm}_{\lambda}$ are coercive on Nehari set.
	\end{proposition} 
	The following lemma is about the behaviour of energy functional with respect to the parameter $\lambda$. 
	\begin{lemma}\label{p:2.10}
		Let $u\in X\setminus \{0\}$. Let $I$ be an open interval in $\mathbb{R}^+$ such that $t^{\pm}_\lambda$ are well defined for all $\lambda\in I$. Then,
		\begin{enumerate}
			\item [(i)]the functions $I\ni \lambda\rightarrow t^{\pm}_\lambda(u)$ are $C^1$. Moreover, $I\ni \lambda\rightarrow t^{-}_\lambda(u)$ is decreasing while $I\ni \lambda\rightarrow t^{+}_\lambda(u)$ is increasing.
			\item [(ii)] the functions $I\ni \lambda\rightarrow {\mathcal J}^{\pm}_{\lambda}(u)$ are $C^1$ and  decreasing.
		\end{enumerate}
	\end{lemma}
	\begin{proof} The proof follows from implicit function theorem.
	\end{proof}
	\begin{remark}
		When $I=(0,\lambda^*_{a, b})$, both the claims in Lemma \ref{p:2.10} remain true, independent of $u\in X$.
	\end{remark}
	
	In order to prove our main result we have designed the following three sections dealing with the case $\lambda<\lambda^*_{a,b}, \lambda=\lambda^*_{a,b}$ and $\lambda>\lambda^*_{a,b}$ respectively. 
	\section{\textbf{Existence of  solutions when $0<\lambda<\lambda^*_{a,b}$}}
	In order to prove Theorem $\ref{T1}$ first we have proved existence of at least two solutions of \eqref{p} for $\lambda\in (0,\lambda^*_{a,b})$. For that,  we begin by showing that the functional $\mathcal E_{\lambda}$ achieves its minimizers in the Nehari decomposition $\mathcal{N}^{-}_{\lambda}$ and $\mathcal{N}^{+}_{\lambda}$ for all $\lambda\in(0,\lambda^*_{a,b})$  for both  degenerate and non-degenerate cases.  We close this section by verifying  that these minimizers are the solutions to the problem in $\eqref{p}$.
	
	\begin{lemma}\label{l3.1}
		For each $0<\lambda<\lambda^*_{a,b}$,	there exists $ v_\lambda \in \mathcal{N}^{+}_{\lambda}$  such that  $\mathcal{J}^{+}_{\lambda}(v_\lambda)=\hat{\mathcal{J}}^{+}_{\lambda}$.
	\end{lemma}
	\begin{proof}
		Let $\{v_n\} \subset \mathcal{N}^{+}_{\lambda}$ be an arbitrary minimizing sequence  for ${\mathcal{J}}^{+}_{\lambda}$, that is ${\mathcal{J}}^{+}_{\lambda}(v_n)\rightarrow \hat{\mathcal{J}}^{+}_{\lambda}$. First we claim that $\{v_n\}$ is bounded. In fact, 	when $v_n\in \mathcal{N}^{+}_{\lambda}$ using  H\"older inequality, we have
		\begin{equation}\label{bound}
			b(p-2\theta)\|v_n\|^{2\theta}_{X}<a(p-2)\|v_n\|^2_{X}+	b(p-2\theta)\|v_n\|^{2\theta}_{X}<
			\lambda(p-\gamma)S^{-\frac{\gamma}{2}} \|f\|_{\frac{2^*_s}{2^*_s -\gamma}}\|v_n\|^{\gamma}_{X},
		\end{equation}
		implies that sequence $\{v_n\}$ is bounded.
		Moreover, one can observe that when $v\in\mathcal{N}^{-}_{\lambda}$, we have
		\begin{align*}
			\mathcal E_{\lambda}(v)&=a\left(\frac{1}{2}-\frac{1}{p}\right)\|v_n\|^2_{X}+b\left(\frac{1}{2\theta}-\frac{1}{p}\right)\|v_n\|^{2\theta}_{X}-\left(\frac{1}{\gamma}-\frac{1}{p}\right)\lambda\int_\Omega f|v_n|^{\gamma}dx\\
			&<			a(p-2)\left(\frac{\gamma-2}{2p\gamma}\right)\|v\|^2_{X}+b(p-2\theta)\left(\frac{\gamma-2\theta}{2p\theta\gamma}\right)\|v\|^{2\theta}_{X}.
		\end{align*} For for non-degenerate case (when $\gamma\in(1,2)$), we get $\mathcal E_{\lambda}(v)<0.$
		and hence,  $\hat{\mathcal{J}}^{+}_{\lambda}<0$. 
		For degenerate case (when $\gamma\in(2,2\theta)$), we get
		\[\mathcal E_{\lambda}(v)<(p-2\theta)\left(\frac{\gamma-2\theta}{2p\theta\gamma}\right)\|v\|^{2\theta}_{X}<0,\] and hence $\hat{\mathcal{J}}^{+}_{\lambda}<0$.
		Moreover,  $\hat{\mathcal{J}}^{+}_{\lambda}>-\infty$ from definition of $\mathcal E_{\lambda}$. Thus  up to a subsequence, $v_n\rightharpoonup v_\lambda\geq 0$ in $X$. Let us prove $v_{\lambda}\not \equiv 0$. If $v_{\lambda}\equiv0$, then $0=\mathcal E_{\lambda}(v_{\lambda})\leq \displaystyle \liminf_{n\to \infty}\mathcal E_{\lambda}(v_n)=\hat{\mathcal{J}}^{+}_{\lambda}<0$, a contradiction and hence, $v_{\lambda}\not \equiv 0$. Observe that, irrespective of $v_\lambda\in \mathcal C^\pm$, there exists $t^{+}_{\lambda}(v)>0$ such that $t^{+}_{\lambda}(v_\lambda)v_\lambda\in \mathcal{N}^{+}_{\lambda}$ and $\psi_{\lambda,v_{\lambda}}$ is decreasing in $(0,t^{+}_{\lambda}(v_\lambda))$  with $\psi'_{\lambda,v_{\lambda}}(t^{+}_{\lambda}(v_\lambda))=0$.
		
		Our claim is $v_n \rightarrow v_{\lambda}$ in $X$. Suppose on contrary $v_n\rightharpoonup v$ then $v_n\rightarrow v$ in $L^q(\Omega)$ where $q\in [1,2^*_{s})$ and  $|v_n(x)|\leq h(x)$ a.e. in $\Omega$ for some $h(x)\in L^q(\Omega)$. Under the assumptions $(F)$ and $(G)$, using H\"older inequality and Lebesgue dominated convergence theorem, we obtain
		\[\int_{ \Omega}f(|v_n|^{\gamma}-|v_\lambda|^{\gamma})\rightarrow0,\;\;
		\int_\Omega g(x)|v_n(x)|^p\rightarrow \int_\Omega g(x)|v_\lambda(x)|^p.\] From weak lower semicontinuity of norm, $\|v_\lambda\|_{X}<\liminf_{n\to\infty}\|v_n\|_{X}$. Using all this information, we have 
		\begin{align*}
			&	\liminf_{n\rightarrow\infty} \psi'_{\lambda,v_n}(t^{+}_{\lambda}(v_{\lambda}))=	\liminf_{n\rightarrow\infty}\left[a(t^{+}_{\lambda})(v_\lambda)\|v_n\|^{2}_{X}+b (t^{+}_{\lambda})^{2\theta-1}(v_\lambda)\|v_n\|^{2\theta}_{X}\right]\\
			&-\liminf_{n\to \infty}\left[(t^{+}_{\lambda}(v_{\lambda}))^{\gamma-1}\lambda\int_{ \Omega}f|v_n|^{\gamma}dx+(t^{+}_{\lambda}(v_{\lambda}))^{p-1}\int_\Omega g(x)|u_n(x)|^p\right]\\
			&>a(t^{+}_{\lambda})(v_\lambda)\|v_\lambda\|^{2}_{X}+b (t^{+}_{\lambda})^{2\theta-1}(v_\lambda)\|v_\lambda\|^{2\theta}_{X}-(t^{+}_{\lambda}(v_{\lambda}))^{\gamma-1}\lambda\int_{ \Omega}f|v_\lambda|^{\gamma}dx\\
			&-(t^{+}_{\lambda}(v_{\lambda}))^{p-1}\int_\Omega g(x)|u_\lambda(x)|^p\\
			&= \psi'_{\lambda,v_\lambda}(t^{+}_{\lambda}(v_{\lambda}))=0.
		\end{align*}

		Therefore, $\psi'_{\lambda,v_n}(t^{+}_{\lambda}(v_{\lambda}))>0$, for sufficiently large $n$. Since $v_n\in \mathcal{N}^{+}_{\lambda}$, by possible fibering map it is easy to see that $\psi'_{\lambda,v_n}(t)<0$ for $0<t<1$ and $\psi '_{\lambda,v_n}(1)=0$ ($t^{+}_{\lambda}(v_n)=1$), therefore $t^{+}_{\lambda}(v_{\lambda})>1=t^{+}_{\lambda}(v_{n})$ for large $n$.
		Thus, we have
		\[\hat{\mathcal{J}}^{+}_{\lambda}\leq{\mathcal E_\lambda}((t^{+}_{\lambda}(v_{\lambda})v_{\lambda}))=\mathcal{J}^{+}_{\lambda}(v_{\lambda})<\liminf_{n\rightarrow\infty}\mathcal{J}^{+}_{\lambda}(v_n)=\hat{\mathcal{J}}^{+}_{\lambda},\]
		which is an absurd. Hence $v_{n}\rightarrow v_{\lambda}$ in $X$. As a consequence of strong convergence, we have
		\[\lim_{n\rightarrow \infty}\psi^{'}_{\lambda,v_n}(1)=\psi^{'}_{\lambda,v_{\lambda}}(1)=0 \ \ \textrm{and} \ \ \lim_{n\rightarrow \infty}\psi^{''}_{\lambda,v_n}(1)=\psi^{''}_{\lambda,v_{\lambda}}(1)\geq 0 ,\]
		and since $\mathcal{N}^{0}_{\lambda}=\emptyset$ for $0<\lambda<\lambda^*_{a,b}$, we have
		$v_{\lambda}\in \mathcal{N}^{+}_{\lambda}$ and $\mathcal{J}^{+}_{\lambda}(v_{\lambda})=\hat{\mathcal{J}}^{+}_{\lambda}.$
	\end{proof}
	
	\begin{lemma}\label{l3.2}
		For each $0<\lambda<\lambda^*_{a,b}$,	there exists $u_\lambda\in \mathcal{N}^{-}_{\lambda}$  such that $\mathcal{J}^{-}_{\lambda}(u_\lambda)=\hat{\mathcal{J}}^{-}_{\lambda}$.
	\end{lemma}
	\begin{proof}  The proof is similar to Lemma \ref{l3.1}. 
	\end{proof}
	In order to show that the minimizers obtained in Lemma \ref{l3.1} and Lemma \ref{l3.2} are the weak solutions of \eqref{p}, we make use of Theorem 2.3 of \cite{kj} as for $\lambda<\lambda_{a,b}$, the set  $\mathcal{N}^{0}_{\lambda}=\emptyset$. Therefore the above  non-trival minimizers of $\mathcal{E}_{\lambda}$ are critical points of $\mathcal{E}_{\lambda}$ in $X$, equivalently weak solutions of \eqref{p}.
	\section{Existence of solutions when  $\lambda=\lambda^*_{a,b}$}
	In this section, we consider the case when the parameter $\lambda$ takes the extremal value, that is, $\lambda=\lambda^*_{a,b}$. Contrary to the case of $\lambda<\lambda^*_{a,b}$, the set $\mathcal N^0_{\lambda^*_{a,b}}$ is non-empty. Hence the study of $\mathcal N^0_{\lambda^*_{a,b}}$ becomes important. In fact, we also explore the intersection of solution set of $(P_{\lambda^*_{a,b}})$ with $\mathcal N^0_{\lambda^*_{a,b}}$ in this section.
	
	\begin{proposition} The set $\mathcal{N}_{\lambda^*_{a,b}}^0\neq\emptyset$  and $\mathcal{N}_{\lambda^*_{a,b}}^0=\{u\in \mathcal{N}_{\lambda^*_{a,b}}\cap \mathcal C^+:\lambda_{a,b}(u)=\lambda^*_{a,b}\}$ . Moreover, $\mathcal{N}_{\lambda}^0\neq\emptyset$ for all $\lambda\geq \lambda^*_{a,b}$.
	\end{proposition}
	\begin{proof}
		From Proposition \ref{luch}, the minimization problem in \eqref{lmini} has a solution in $X_0\cap \mathcal C^+$, that is, 
		$\lambda_{a,b}(u)=\lambda^*_{a,b}$ for some $u\in X\cap \mathcal C^+$ . Using Proposition \ref{p,2.2} there exists $t_{\lambda^*_{a,b}}^0=t_{\lambda^*_{a,b}}^0(u)>0$ such that $t_{\lambda^*_{a,b}}^0u\in \mathcal{N}^{0}_{\lambda^*_{a,b}=\lambda_{a,b}(u)}$ and hence it's nonempty.
		Since $\lambda_{a,b}(u)$ is unbounded above, therefore for all $\lambda \geq \lambda^*_{a,b}$ there exists $u\in \mathcal{C}^+$ such that $\lambda=\lambda_{a,b}(u)$ and by definition of $\lambda_{a,b}(u)$,  $t(u)u\in \mathcal{N}^0_{\lambda_{a,b}(u)=\lambda}$, hence $\mathcal{N}^{0}_{\lambda_{a,b}} \neq \emptyset$ for all $\lambda\geq \lambda^*_{a,b}$. 
	\end{proof}
	Next we are going to prove an important Lemma which will help us to conclude that the minimizers of the energy functional are not in $\mathcal{N}^{0}_{\lambda^*_{a,b}}$, in spite of $\mathcal{N}^{0}_{\lambda^*_{a,b}}$ being non-empty. 
	\begin{proposition}\label{cor,2.7} For any $a>0$ and $b>0$,	the problem $(P_{\lambda^*_{a,b}})$ has no solution in $ \mathcal{N}^{0}_{\lambda^*_{a,b}}$. 
	\end{proposition}
	\begin{proof}
		In order to give a proof of this Proposition, we first prove that for each $u\in \mathcal{N}^{0}_{\lambda^*_{a,b}}$, we have
		\begin{equation*}
			\begin{aligned}
				& 	-2(a+\theta b \|u\|^{2\theta-2}_{X})(-\Delta)^su-\lambda^*_{a,b} \gamma f |u|^{\gamma-2} u-pg|u|^{p-2}u=0.
			\end{aligned}
		\end{equation*}
		In fact, for $t(u)u\in \mathcal{N}_{\lambda}$, we have
		\[at^2\|u\|^2_{X}+bt^{2\theta}\|u\|^{2\theta}_{X}-{\lambda}t^\gamma\int_\Omega f  |u|^{\gamma} dx - t^{p}\int_\Omega g|u|^p dx=0.\]
		Differentiating above equation with respect to $u$ then for all $w\in X$, we 
		\begin{align*}t'(u)&(
			2at(u)\|u\|^2_{X}+2\theta b (t(u))^{2\theta-1}\|u\|^{2\theta}_{X}-\lambda (u)\gamma(t(u))^{\gamma-1}\int_{ \Omega}f|u|^{\gamma}dx-p(t(u))^{p-1}\int_\Omega g|u|^p dx)\\
			&-2a(t(u))^2 (-\Delta)^suw-2\theta b (t(u))^{2\theta}\|u\|^{2\theta-2}_{X}(-\Delta)^suw\\&-\lambda(u)\gamma(t(u))^{\gamma}\int_{ \Omega}f|u|^{\gamma-2}uwdx-(t(u))^{p}p\int_{ \Omega}g|u|^{p-2}uwdx=0.
		\end{align*}
		As $u\in \mathcal{N}^{0}_{\lambda^*_{a,b}=\lambda(u)}$ implies $t(u)=1$, we get
		\[-2(a+\theta b \|u\|^{2\theta-2}_{X})(-\Delta)^su-\lambda^*_{a,b} \gamma f |u|^{\gamma-2} u-pg|u|^{p-2}u=0 .\]
		Now if $u\in X$ solves the problem $(P_{\lambda^*_{a,b}})$, we have 
		\[-(a+b\|u\|^{2\theta-2}_{X})(-\Delta)^su-\lambda^*_{a,b} f|u|^{\gamma-2}u-g|u|^{p-2}u=0.\] Now solving last two equations by eliminating $\|u\|^{2\theta-2}_{X}$, we get
		\[- 2a(\theta-1)(-\Delta)^su=\lambda^*_{a,b}(2\theta-\gamma)f |u|^{\gamma-2}u+(2\theta-p)g|u|^{p-2}u,\]
		which implies,\[-(-\Delta)^s u=\lambda^*_{a,b}\frac{2\theta-\gamma}{2a(\theta-1)}f|u|^{\gamma-2}u+\frac{2\theta-p}{2a(\theta-1)}g|u|^{p-2}u.\]
		Therefore near $\partial\Omega$, we have $-(-\Delta)^s u\geq 0$ leading to 
		\begin{align*}
			\lambda^*_{a,b}\gamma f|u|^{\gamma-2}u+p g |u|^{p-2}u&=- \left(2a+2b\theta\|u\|^{2\theta-2}_{X}\right)(-\Delta)^s u\\
			&\geq -(2a+2b\|u\|^{2\theta-2}_{X})(-\Delta)^s u\\&=2 \lambda^*_{a,b}f |u|^{\gamma-2}u+2g|u|^{p-2}u
		\end{align*}
		or
		\[\left(\frac{p-2}{2-\gamma}\right)g|u|^{p-\gamma}\geq\lambda^*_{a,b} f>0.\] Now using $u(x)\rightarrow 0$ as $x\rightarrow\partial \Omega$, left hand side of above inequality goes to zero while right hand side is bounded away from zero near boundary in the light of the assumption $(F)$ to give an absurd.
	\end{proof}
	\begin{remark}
		In degenerate Kirchhoff  case, after elimination of $(-\Delta)^s$ term, we get  \begin{equation}\label{eq4.4}
			{ g|u|^{p-\gamma}=\lambda^*\frac{2\theta-\gamma}{p-2\theta} f .}
		\end{equation}	Now, using $(F)$ and choosing  $x$ very close to boundary of $\Omega$, one can observe  that left hand side of \eqref{eq4.4} goes to 0, while the right hand side of \eqref{eq4.4} is finite and bounded away from zero, which gives us an absurd.
	\end{remark}
	As a consequence of above Proposition, we have the following corollary.
	\begin{corollary}
		The set $\mathcal{N}_{\lambda^*_{a,b}}^0$ is compact.
	\end{corollary}
	In order to find solutions of $(P_{\lambda^*})$ we take a sequence $\lambda_n$ such that $\lambda_{n}\uparrow\lambda^*_{a,b}$. As $(P_{\lambda_{n}})$ has solutions $u_{\lambda_n}\in \mathcal{N}^{-}_{\lambda_{n}}$ and $v_{\lambda_{n}}\in \mathcal{N}^{+}_{\lambda_{n}}$ we expect  strong limits of these solution sequences to give rise to solutions of $(P_{\lambda^*_{a,b}})$.
	Based on the definition in equation \eqref{capnl}, we have the following couple of propositions which are required to study the limiting behaviour in case of $\lambda\uparrow \lambda^*_{a,b}$. 
	
	\begin{proposition}\label{prop 4.2} There holds,
		\[\overline{\mathcal{N}^{+}_{\lambda^*_{a,b}}\cup \mathcal{N}^{-}_{\lambda^*_{a,b}}}=\mathcal{N}^{+}_{\lambda^*_{a,b}}\cup\mathcal{N}^{-}_{\lambda^*_{a,b}}\cup\mathcal{N}^{0}_{\lambda^*_{a,b}}\cup \{0\}.\]
	\end{proposition}
	\begin{proof}  The proof of this proposition follows by taking a sequence
		$v_n \in \mathcal{N}^{-}_{\lambda^*_{a,b}}$ converging strongly to $v$ in $X$ and using compact sobolev embeddings. 
	\end{proof}
	\begin{remark}  Remark \ref{remc} together with Proposition \ref{prop 4.2} leads to the following observation
		\[\overline{\mathcal{\hat{N}}_{\lambda^*_{a,b}}\cup \hat {\mathcal{N}}^{+}_{\lambda^*_{a,b}}}=\mathcal{\hat{N}}_{\lambda^*_{a,b}}\cup \hat {\mathcal{N}}^{+}_{\lambda^*_{a,b}}\cup \{tu:t>0, t\in \mathcal{N}^{0}_{\lambda^*_{a,b}}\}\cup \{0\}.\]
	\end{remark}

	Now we have necessary mathematical background to show the existence of at least two positive solutions of the problem $(P_{\lambda^*_{a,b}})$ taking advantage of multiplicity of solutions \eqref{p} when $\lambda<\lambda^*_{a,b}$ via limiting behaviour.

	Define the map ${t_{\lambda^*_{a,b}}}:\overline{\hat{\mathcal{N}}_{\lambda^*_{a,b}}}\setminus\{0\}\rightarrow\mathbb{R}$ as
	
	\[t_{\lambda^*_{a,b}}(u)=
	\begin{cases} t^{-}_{\lambda^*_{a,b}}(u)& u\in \hat{\mathcal{N}}_{\lambda^*_{a,b}} \\ t^{0}_{\lambda^*_{a,b}}(u) & \textrm{otherwise}
	\end{cases}\] 
	
	and $s_{\lambda^*_{a,b}}:\overline{\hat{\mathcal{N}}_{\lambda^*_{a,b}}\cup\hat{\mathcal{N}}^{+}_{\lambda^*_{a,b}}}\rightarrow\mathbb{R}$ by
	\[s_{\lambda^*_{a,b}}(v)=
	\begin{cases} t^{+}_{\lambda^*_{a,b}}(v)& v\in \hat{\mathcal{N}}_{\lambda^*_{a,b}}\cup\hat{\mathcal{N}}^{+}_{\lambda^*_{a,b}} \\ t^{0}_{\lambda^*_{a,b}}(v) & \textrm{otherwise}.
	\end{cases}.\]
	The following result can be adopted from \cite{Silva}.
	
	\begin{proposition}\label{p5.1} The following conclusions hold:
		
		\begin{enumerate}
			\item [(i)]
			The functions $s_{\lambda^*_{a,b}}$ and $t_{\lambda^*_{a,b}}$ are continuous. Moreover,
			once $u\notin \hat{\mathcal{N}}^{+}_{\lambda^*_{a,b}}$, we have 
			\[\lim_{\lambda \uparrow \lambda^*_{a,b}}t^{-}_{\lambda}(u)=t_{\lambda^*_{a,b}}(u), \ \  \lim_{\lambda \uparrow \lambda^*_{a,b}}t^{+}_{\lambda}(u)=s_{\lambda^*_{a,b}}(u) \]
			and
			\[\lim_{\lambda \uparrow \lambda^*_{a,b}} \mathcal E_{\lambda}(t^{-}_\lambda(u)u)=\mathcal E_{{\lambda}^{*}_{a,b}}(t_{\lambda^*_{a,b}}(u)u), \ \   \lim_{\lambda \uparrow \lambda^*_{a,b}} \mathcal E_{\lambda}(t^{+}_\lambda(u)u)=\mathcal E_{{\lambda}^{*}_{a,b}}(s_{\lambda^*_{a,b}}(u)u).\]
			
			\item[(ii)] The set $\mathcal{N}^{0}_{\lambda^*_{a,b}}$ has empty interior in $\mathcal{N}_{\lambda^*_{a,b}}$.
		\end{enumerate}
		
	\end{proposition}
	
	Now we are ready to discuss following minimization problems
	\[\hat{\mathcal E}^{-}_{\lambda^*_{a,b}}=\inf\{\mathcal E_{\lambda^*_{a,b}}(t_{\lambda^*_{a,b}}(u)u):u\in \mathcal{N}^{-}_{\lambda^*_{a,b}} \cup \mathcal{N}^{0}_{\lambda^*_{a,b}} \}, \]
	{and}\[
	\hat{\mathcal E}^{+}_{\lambda^*_{a,b}}=\inf\{\mathcal E_{\lambda^*_{a,b}}(s_{\lambda^*_{a,b}}(v)v):v\in \mathcal{N}^{+}_{\lambda^*_{a,b}} \cup \mathcal{N}^{0}_{\lambda^*_{a,b}} \}.\]

	From Proposition \ref{p5.1} it follows that $\hat{\mathcal{J}}^{\pm}_{\lambda^*_{a,b}}=\hat{\mathcal E}^{\pm}_{\lambda^*_{a,b}}$. Next we observe  behaviour of $(0,\lambda^*_{a,b}]\ni\lambda\mapsto \hat{\mathcal{J}}^{\pm}_{\lambda}$    in following proposition when $\lambda\uparrow\lambda^*_{a,b}$.
	
	\begin{proposition}\label{p5.2}
		The function $\lambda \mapsto \hat{\mathcal{J}}^{\pm}_{\lambda}$ is decreasing for all $\lambda\in(0,\lambda^*_{a,b}].$ Moreover,
		\[\lim_{\lambda \uparrow \lambda^*_{a,b}} \hat{\mathcal{J}}^{\pm}_{\lambda}=\hat{\mathcal{J}}^{\pm}_{\lambda^*_{a,b}}.\]
	\end{proposition}
	\begin{proof}
		We know from Lemma \ref{p:2.10} that $\mathcal{J}^{\pm}_{\lambda}$ is decreasing   on $(0,\lambda^*_{a,b})$, from there we can conclude that $\hat{\mathcal{J}}^{\pm}_{\lambda}$ is decreasing on $(0,\lambda^*_{a,b})$. To  show  that $\mathcal{J}^{\pm}_{\lambda}$  is decreasing  on $(0,\lambda^*_{a,b}]$ take $\lambda\in(0,\lambda^*_{a,b})$ arbitrary. Then for all $u\in \mathcal{N}^{-}_{\lambda^*_{a,b}}\cup \mathcal{N}^{0}_{\lambda^*_{a,b}}$ using decreasing behaviour of $\mathcal{J}^{-}_{\lambda}$ and Proposition \ref{p5.1}, we have 
		\begin{align*}
			\hat{\mathcal{J}}^{-}_{\lambda^*_{a,b}}=\hat{\mathcal E}^{-}_{\lambda^*_{a,b}}\leq \mathcal E_{\lambda^*_{a,b}}(t_{\lambda^*_{a,b}}(u)u)&={\lim_{{\Lambda} \downarrow \lambda^*_{a,b}}}\mathcal E_{\Lambda}(t^{-}_{{\Lambda}}(u)u)=\lim_{{\Lambda} \downarrow \lambda^*_{a,b}}\mathcal{J}^{-}_{{\Lambda}}(u)\\
			&<\lim_{{\Lambda} \downarrow \lambda^*_{a,b}}\mathcal{J}^{-}_{\lambda^*_{a,b}}(u)=\mathcal{J}^{-}_{\lambda^*_{a,b}}(u)<\mathcal{J}^{-}_{\lambda}(u),
		\end{align*}
		and hence $\hat{\mathcal{J}}^{-}_{\lambda^*_{a,b}}\leq \hat{\mathcal{J}}^-_{\lambda}$. Similarly it follows that $\hat{\mathcal{J}}^{+}_{\lambda^*_{a,b}}\leq \hat{\mathcal{J}}^+_{\lambda}$.
		
		To prove remaining  part take $\lambda_{n}\in(0,\lambda^*_{a,b}]$ such that $\lambda_n \uparrow \lambda^*_{a,b}$. Then $\hat{\mathcal{J}}^{-}_{\lambda_{n}}\geq \hat{\mathcal{J}}^{-}_{\lambda^*_{a,b}}$. Assume on contrary $\hat{\mathcal{J}}^{-}_{\lambda_{n}} \rightarrow J> \hat{\mathcal{J}}^{-}_{\lambda^*_{a,b}}$
		and suppose that there exists $\delta>0$ such that,   $J-\hat{\mathcal{J}}^{-}_{\lambda^*_{a,b}}\geq \delta$. Choose $\beta>0$ such that $2\beta<\delta$ and  $u(\beta)\in \mathcal{N}^{-}_{\lambda^*_{a,b}}$ such that $\mathcal{J}^{-}_{\lambda^*_{a,b}}(u(\beta))-\hat{\mathcal{J}}^{-}_{\lambda^*_{a,b}}\leq \beta$. Now, using continuity of $\mathcal{J}^{-}_{\lambda}$  when $\lambda<\lambda^*_{a,b}$, we have
		\[0\leq \mathcal{J}^{-}_{\lambda_{n}}(u(\beta))-\mathcal{J}^{-}_{\lambda^*_{a,b}}(u(\beta))\leq \beta\]
		for large  values of $n$. Thus,
		\begin{align*}
			\hat{\mathcal{J}}^{-}_{\lambda_{n}}&\leq \mathcal{J}^{-}_{\lambda_{n}}(u(\beta)) 
			\leq \mathcal{J}^{-}_{\lambda^*_{a,b}}(u(\beta))+\beta  
			\leq \hat{\mathcal{J}}^{-}_{\lambda^*_{a,b}}+2\beta 
			\leq J-\delta+2\beta < J. 
		\end{align*}
		
		Therefore,  for $n$  sufficiently large we have  $J<J$, an absurd. Hence $J=\hat{\mathcal{J}}^{-}_{\lambda^*_{a,b}}$. {Similarly, we can show that 	$\lim_{\lambda \uparrow \lambda^*_{a,b}} \hat{\mathcal{J}}^{+}_{\lambda}=\hat{\mathcal{J}}^{+}_{\lambda^*_{a,b}}.$}
	\end{proof}
	
	\begin{lemma}\label{prop5.6}
		There exists $u_{\lambda^*_{a,b}}\in \mathcal{N}^{-}_{\lambda^*_{a,b}}$ and $v_{\lambda^*_{a,b}}\in \mathcal{N}^{+}_{\lambda^*_{a,b}}$ such that $\hat{\mathcal{J}}^{-}_{\lambda^*_{a,b}}=\mathcal{J}^{-}_{\lambda^*_{a,b}}(u_{\lambda^*_{a,b}})$ and $\hat{\mathcal{J}}^{+}_{\lambda^*_{a,b}}=\mathcal{J}^{+}_{\lambda^*_{a,b}}(v_{\lambda^*_{a,b}})$.
	\end{lemma}
	\begin{proof}
		We already  know from Lemma \ref{l3.1}, \ref{l3.2} that we have $u_{\lambda_n}\in \mathcal{N}^{-}_{\lambda_n}$  and $v_{\lambda_n}\in \mathcal{N}^{+}_{\lambda_n}$ such that $\mathcal{J}^{-}_{\lambda}(u_{\lambda_n})=\hat{\mathcal{J}}^{-}_{\lambda_n}$ and $\mathcal{J}^{+}_{\lambda_n}(v_{\lambda_n})=\hat{\mathcal{J}}^{+}_{\lambda_n}$ respectively when $\lambda_n<\lambda^*_{a,b}$  for each $n\in \mathbb{N}$. To show the existence of minimizers for $\mathcal E_{\lambda^*_{a,b}}$ we are looking to pass the limit $\lambda_n\uparrow\lambda^*_{a,b}$ in $\hat {\mathcal J}_{\lambda_n}^\pm$.  We know for each $\lambda_{n} <\lambda^*_{a,b}$, we can get a sequence   $\{u_{\lambda_n}\}\subset \mathcal{N}_{\lambda_{n}}^{-}$ such that $\hat{\mathcal{J}}^{-}_{\lambda_{n}}=\mathcal{J}^{-}_{\lambda_{n}}(u_{\lambda_n})$. Moreover, $\{u_{\lambda_n}\}$ is bounded, otherwise, if $\|u_{\lambda_{n}}\|_X \rightarrow\infty$, then  
		$\infty>\lim \hat{\mathcal{J}}^{-}_{\lambda_{n}}\geq \infty$, which is an absurd.
		Therefore, up to a subsequence, we get
		\[ \left\{
		\begin{array}{ll}
			u_{\lambda_n} \rightharpoonup u_{\lambda^*_{a,b}} \hspace{1.3cm}   \ \textrm{in} \ X, \\
			u_{\lambda_n} \rightarrow u_{\lambda^*_{a,b}} \ \textrm{in} \ \ \ \  \ \ \  \ \textrm{in} \ L^q(\Omega)  \  \forall \ {q\in[1,2^*_{ s})},\\
			u_{\lambda_n} \rightarrow u_{\lambda^*_{a,b}} \hspace{1.5cm}  \textrm{a.e.} \ \textrm{in} \ \Omega\\
		\end{array}
		\right.
		\] 
		with $u_{\lambda^*_{a,b}}\geq0$. Our claim is $u_{\lambda^*_{a,b}}\not\equiv0$ and $u_{\lambda^*_{a,b}}\in \mathcal C^+$. Otherwise, we would have $\|u_{\lambda_n}\|_{X}\rightarrow0$, an absurd.   {Now, once $u_{\lambda_{n}}$ is solution of $(P_{\lambda_{n}})$, for all $\phi\in X$, we have
			\[		(a+b\|u_{\lambda_n}\|^{2(\theta-1)}_{X})\iint_{\mathbb R^{2N}}  \frac{(u_{\lambda_n}(x)-u_{\lambda_n}(y))(\phi(x)-\phi(y))}{|x-y|^{N+2s}}dxdy\]\[-\lambda_{n}\int_{ \Omega}f u_{\lambda_{n}}^{\gamma-1}\phi dx-\int_{ \Omega}gu^{p-1}_{\lambda_n}\phi dx=0.\]Taking $u_{\lambda_n} - u_{\lambda^*_{a,b}}$ as a text function in above equation and suppose $\|u_{\lambda_n}\|_{X}\rightarrow \alpha>0$, we get
			\begin{align*}
				(a+&b\|u_{\lambda_n}\|^{2(\theta-1)}_{X})\iint_{\mathbb R^{2N}}  \frac{|u_{\lambda_n}(x)-u_{\lambda_n}(y)|^2}{|x-y|^{N+2s}}dxdy\\ &-(a+b\|u_{\lambda_n}\|^{2(\theta-1)}_{X})\iint_{\mathbb R^{2N}}  \frac{(u_{\lambda_n}(x)-u_{\lambda_n}(y))(u_{\lambda^*_{a,b}}(x)-u_{\lambda^*_{a,b}}(y))}{|x-y|^{N+2s}}dxdy\\&-\lambda_{n}\int_{ \Omega}f u_{\lambda_{n}}^{\gamma-1}(u_{\lambda_n}(x)-u_{\lambda^*_{a,b}}(x)) dx-\int_{ \Omega}gu^{p-1}_{\lambda_n}(u_{\lambda_n}-u_{\lambda^*_{a,b}})(x)dx=0.\end{align*}
			Using weak lower semi-continuity of norms, we get
			\begin{align*}
				&	(a+b\alpha^{2(\theta-1)})(\lim_{n\rightarrow\infty}\|u_{\lambda_n}\|_{X}^2-\|u_{\lambda^*_{a,b}}\|_{X}^2)=0.\end{align*}
			Therefore, $\lim_{n\rightarrow\infty} \|u_{\lambda_n}\|^{2}_{X}=\|u_{\lambda^*_{a,b}}\|^2_{X}$ ( because $a+b\alpha^{2\theta-2}>0$), which implies that $u_{\lambda_n}\rightarrow u_{\lambda^*_{a,b}}$ in $X$}.
		From strong convergence, we have 
		\[\psi'_{\lambda^*,u_{\lambda^*_{a,b}}}(1)=
		\lim_{n\to \infty}\psi'_{\lambda_{n},u_{\lambda_{n}}}(1)=0 \ \ \textrm{and} \ \ \psi''_{\lambda^*_{a,b},u_{\lambda^*_{a,b}}}(1)=\lim_{n\to \infty}\psi''_{\lambda_{n},u_{\lambda_{n}}}(1)\leq0.\]
		Also when $u_{\lambda_n}\in \mathcal{N}^{-}_{\lambda_n}$, we have
		\[0<b(2\theta-\gamma)\|u_{\lambda^*_{a,b}}\|^{2\theta}_{X}=b(2\theta-\gamma)\lim_{n\rightarrow\infty}\|u_{\lambda_n}\|^{2\theta}_{X}\leq(p-\gamma)\int_{ \Omega}gu^p_{\lambda^*_{a,b}}dx,\]
		thus, $u_{\lambda^*_{a,b}}\in \mathcal{N}^{-}_{\lambda^*_{a,b}}\cup \mathcal{N}^{0}_{\lambda^*_{a,b}}. $ From strong convergence of $\{u_{\lambda_{n}}\}$ and Proposition \ref{p5.2}, we have
		\[\mathcal E_{\lambda^*_{a,b}}(u_{\lambda^*_{a,b}})=\lim_{n\rightarrow\infty}\mathcal E_{\lambda_{n}} (u_{\lambda_{n}})=\lim_{\lambda_{n}\uparrow\lambda^*_{a,b}}\hat{\mathcal{J}}_{\lambda_{n}}^-=\hat{\mathcal{J}}^{-}_{\lambda^*_{a,b}}.\] We claim that $u_{\lambda^*_{a,b}}\not\in \mathcal{N}^{0}_{\lambda^*_{a,b}}$. If not, we get a contradiction from  Proposition \ref{cor,2.7}.
		Therefore $u_{\lambda^*_{a,b}}$ is a constrained minimizer of $\mathcal{E}_{\lambda^*_{a,b}}$ forced to be in $\mathcal{N}^{-}_{\lambda^*_{a,b}}$.
		
		In order to achieve second minimizer $v_{\lambda^*_{a,b}}$ of the energy functional $\mathcal{E}_{\lambda^*_{a,b}}$ , taking sequence $\{v_{\lambda_{n}}\} \subset \mathcal{N}^{+}_{\lambda_{n}}$  where $\lambda_{n}\uparrow\lambda^*_{a,b}$ and following same arguments as above we get the second  minimizer $ v_{\lambda^*_{a,b}}$  of $\mathcal{E}_{\lambda^*_{a,b}}$ such that 
		\[\mathcal E_{\lambda^*_{a,b}}(v_{\lambda^*_{a,b}})=\lim_{n\rightarrow\infty} \mathcal E_{\lambda}(v_{\lambda_n})=\lim_{n\rightarrow\infty} \hat{{\mathcal{J}}}^{+}_{\lambda_{n}}=\hat{\mathcal{J}}^{+}_{\lambda^*_{a,b}},\]
		forced to be in $\mathcal{N}^{+}_{\lambda^*_{a,b}}$ following the same argument as above. This completes the proof.
		
	\end{proof} 
	In this case, i.e., $\lambda=\lambda^*_{a,b}$  the non-trivial minimizers $u_{\lambda^*_{a,b}},v_{\lambda^*_{a,b}}$ obtained in Lemma \ref{prop5.6} are avoided to be in $\mathcal{N}^{0}_{\lambda^*_{a,b}}$ as in Proposition \ref{cor,2.7} and hence again from Theorem 2.3 of \cite{kj} are non trivial weak solutions of $(P_{\lambda^*_{a,b}})$. 
	
	\section{\textbf{Existence of  solutions when $\lambda> \lambda^*_{a,b}$}} Since, $\mathcal N_{\lambda^*_{a,b}}^0\not\equiv \emptyset$, the minimizers in $\mathcal N_{\lambda}^\pm$ may not be the critical points of the associated energy functional of the problem \eqref{p}. Therefore, we look for the minimizers of associated energy functional over  subsets of $\mathcal{N}^{+}_{\lambda^*_{a,b}}$ and $\mathcal{N}^{-}_{\lambda^*_{a,b}}$ (defined below) which are strictly separated from  $\mathcal{N}^{0}_{\lambda^*_{a,b}}$. Subsequently, projections of  these minimizers lying in $\mathcal N_{\lambda}^\pm$ for $\lambda\in (\lambda^*_{a,b}, \lambda^*_{a,b}+\epsilon)$ for sufficiently small $\epsilon>0$ turn out to be the desired critical points.

	\begin{proposition} \label{prml}	Let $0<c^+<c^-$, $a\geq0, b>0$ and $\lambda_{n}\downarrow\lambda^*_{a,b}$
		\begin{itemize}
			\item [(i)]
			if $u_{n}\in \mathcal{N}^{-}_{\lambda^*_{a,b}}$ which satisfy $c^+\leq \|u_n\|_{X}\leq c^-$ for all $n\in \mathbb N$ and
			\begin{align*}
				a(t_{\lambda_{n}}^{-}(u_n))^{2}\|u_n\|^{2}_X&+b	(2\theta-1)(t_{\lambda_{n}}^{-}(u_n))^{2\theta}\|u_n\|_{X}^{2\theta}-\lambda(\gamma-1)(t_{\lambda_{n}}^{-}(u_n))^{\gamma}\\
				&\lambda_n \int_{ \Omega}f|u_{n}|^{\gamma}dx-(p-1)(t_{\lambda_{n}}^{-}(u_n))^{p}\int_\Omega g|u_n|^pdx\rightarrow 0 \end{align*} 
			as $n\rightarrow\infty$, then $\textrm{dist}(u_n,\mathcal{N}^{0}_{\lambda^*_{a,b}})\rightarrow0$.
			\item [(ii)]
			if $v_{n}\in \mathcal{N}^{+}_{\lambda^*_{a,b}}$ which satisfy $c^+\leq \|v_n\|_{X}\leq c^-$ for all $n\in \mathbb N$ and
			\begin{align*}
				a(t_{\lambda_{n}}^{+}(v_n))^{2}\|v_n\|^{2}_X&+b(2\theta-1)(t_{\lambda_{n}}^{+}(v_n))^{2\theta}\|v_n\|_{X}^{2\theta}-\lambda(\gamma-1)(t_{\lambda_{n}}^{+}(v_n))^{\gamma}\\&\lambda_n \int_{ \Omega}f|v_{n}|^{\gamma}dx-(p-1)(t_{\lambda_{n}}^{-}(v_n))^{p}\int_\Omega g|v_n|^pdx\rightarrow 0 \end{align*} 
			as $n\rightarrow\infty$, then $\textrm{dist}(v_n,\mathcal{N}^{0}_{\lambda^*_{a,b}})\rightarrow0$.
		\end{itemize}
	\end{proposition}
	\begin{proof} $(i)$  For $u_n\in \mathcal{N}^{-}_{\lambda^*_{a,b}}$, we have $ \displaystyle \int_\Omega g|u_n|^pdx\geq c^+$. Our claim is $\int_{ \Omega}f|u_n|^{\gamma}dx\geq c^+$. 
		Therefore   from Proposition {\ref{p,2.2}}, there exist two critical points   $t_{\lambda_n}^{-}=t_{\lambda_n}^{-}(u_n)$ and $t_{\lambda_{n}}^{+}(u_n)=t_{\lambda_n}^{+}$  for $\psi_{\lambda_{n},u_n}$ for large $n$  
		satisfying $t_{\lambda_{n}}^{+}(u_n)<t_{\lambda_{n}}^{-}(u_n)$ such that $t_{\lambda_{n}}^{+}u_n\in \mathcal{N}_{\lambda_{n}}^{+}$ and $t_{\lambda_{n}}^{-}u_n\in \mathcal{N}_{\lambda_{n}}^{-}$. That is,
		\begin{equation}\label{three}	\left\{
			\begin{aligned}&a(t_{\lambda_{n}}^{-})^{2}\|u_n\|^{2}_X+b(t_{\lambda_{n}}^{-})^{2\theta}\|u_n\|_{X}^{2\theta}-\lambda_n(t_{\lambda_{n}}^{-})^{\gamma}\int_{ \Omega}f|u_n|^{\gamma}dx\\&\quad-(t_{\lambda_{n}}^{-})^{p}\int_{ \Omega}g|u_n|^pdx=0,\\
				&a(t_{\lambda_{n}}^{-})^{2}\|u_n\|^{2}_X+b(2\theta-1)(t_{\lambda_{n}}^{-})^{2\theta}\|u_n\|_{X}^{2\theta}- \lambda_n(\gamma-1)(t_{\lambda_{n}}^{-})^{\gamma}\int_{ \Omega}f|u_n|^\gamma dx\\&\quad - (t_{\lambda_{n}}^{-})^{p}(p-1)\int_{ \Omega}g(x)|u_n|^pdx=o(1),\\&
				a(t_{\lambda_{n}}^{+})^{2}\|u_n\|^{2}_X+b(t_{\lambda_{n}}^{+})^2\|u_n\|_{X}^2-\lambda_n(t_{\lambda_{n}}^{+})^{\gamma}\int_{ \Omega}f|u_n|^{\gamma}dx\\&\quad-(t_{\lambda_{n}}^{+})^{p}\int_{ \Omega}|u_n|^pdx=0.\end{aligned}
			\right. 
		\end{equation}
		Getting the values of $\int_{ \Omega}f |u_n|^{\gamma}dx$ and $\displaystyle \int_\Omega g(x)|u_n|^pdx$   by solving first and third equations of (\ref{three}) and substituting the values in second equation of \eqref{three} and using simple calculus, we get
		\begin{equation}\label{nm}
			\begin{aligned}
				a\|u_n\|_{X}^{2}(t^-_{\lambda_n})^{2}&\left(1-(\gamma-1)\left(\frac{1-\left(\frac{t^{+}_{\lambda_{n}}}{t^{-}_{\lambda_{n}}}\right)^{2-p} }{1-\left(\frac{t^{+}_{\lambda_{n}}}{t^{-}_{\lambda_{n}}}\right)^{\gamma-p}}\right)-{(p-1)}\left(\frac{1-\left(\frac{t^{+}_{\lambda_{n}}}{t^{-}_{\lambda_{n}}}\right)^{2-\gamma} }{1+\left(\frac{t^{+}_{\lambda_{n}}}{t^{-}_{\lambda_{n}}}\right)^{p-\gamma}}\right)\right)\\&+
				b\|u_n\|_{X}^{2\theta}(t^-_{\lambda_n})^{2\theta}\left((2\theta-1)-(\gamma-1)\left(\frac{1-\left(\frac{t^{+}_{\lambda_{n}}}{t^{-}_{\lambda_{n}}}\right)^{2\theta-p} }{1-\left(\frac{t^{+}_{\lambda_{n}}}{t^{-}_{\lambda_{n}}}\right)^{\gamma-p}}\right)\right.\\&\left.\hspace{3.5cm}-{(p-1)}\left(\frac{1-\left(\frac{t^{+}_{\lambda_{n}}}{t^{-}_{\lambda_{n}}}\right)^{2\theta-\gamma} }{1-\left(\frac{t^{+}_{\lambda_{n}}}{t^{-}_{\lambda_{n}}}\right)^{p-\gamma}}\right)\right) =o(1).
			\end{aligned}
		\end{equation}
		{Since both term in left hand side of above equation are positive and entire term converges to zero implies both the term will converge to zero separately.} {Suppose $t^-_{\lambda_n}\rightarrow\alpha$, $t^+_{\lambda_n}\rightarrow\beta$. Taking limit $n\rightarrow\infty$ in \eqref{nm}}  {we get, $t_{\lambda_{n}}^{-}(u_n)$$\setminus$  $t_{\lambda_{n}}^{+}(u_n)\rightarrow 1$,i.e. $\alpha=\beta$, as $1$ is the only zero of 
			\[m(t)=(2p-2\theta-2)t^{\gamma-p}+(\gamma-p)t^{2-p}+(\gamma-p)t^{2\theta-p}+2-2\gamma+2\theta.\]} Once $t^{+}_{\lambda_{n}} u_n\in \mathcal{N}^{+}_{\lambda_n}$, {from \eqref{bound} we have  $\int_{ \Omega}f|u_n|^{\gamma}\geq c$}.  Thus
		\begin{align*}a\|{\alpha}u_n\|^2_{X}+b\|{\alpha}u_n\|_{X}^{2\theta}-\lambda^*_{a,b}\int_{\Omega}f| {\alpha}u_n|^{\gamma}dx-\displaystyle \int_{ \Omega}g(x)|{\alpha}u_n|^pdx=o(1),\\
			a\|{\alpha}u_n\|^2_{X}+b(2\theta-1)\|{\alpha}u_n\|_{X}^{2\theta}-\lambda^*_{a,b}(\gamma-1) \int_{ \Omega}f| {\alpha} u_{n}|^{\gamma}dx&\\
			\quad-(p-1) \displaystyle \int_{ \Omega}g(x)|{\alpha} u_n|^pdx
			&=o(1). \end{align*}
		Solving these two we get
		\[\frac{a(p-2)\|{\alpha} u_n\|^2_{X}+b(p-2\theta)\|{\alpha} u_n\|^{2\theta}_{X}}{(p-\gamma)f |{\alpha} u_n|^{\gamma}dx}=\lambda^*_{a,b}+o(1), \]\[ 
		\frac{a(\gamma-2)\|{\alpha} u_n\|^2_{X}+b(\gamma-2\theta)\|{\alpha} u_n\|^{2\theta}_{X}}{ (\gamma-p)\displaystyle \int_{ \Omega}g(x)|{\alpha} u_n|^pdx}=1+o(1).\]
		Therefore, from expression of $\lambda(u)$, we have
		\[\lambda({\alpha} u_n)=\lambda(u_n)=(1+o(1))^\frac{2\theta-\gamma}{p-2\theta}(\lambda^*+o(1))\]
		leading to $\lambda(u_n)\rightarrow\lambda^*$. Hence $u_n$ is a bounded minimizing sequence for $\lambda^*$. Now, by following same argument as in proof of Proposition \ref{luch}, up to a subsequence,  we obtain $u_n\rightarrow u \in \mathcal{N}^{0}_{\lambda^*}$ and thus $\textrm{dist}(u_n,\mathcal{N}^0_{\lambda^*})\rightarrow0$ as $n\rightarrow\infty$. Similarly from expression of $\lambda_{a,b}(u_n)$, we have $\lambda_{a,b}(u_n)\rightarrow\lambda^*_{a,b}$ and thus $\textrm{dist}(u_n, \mathcal{N}^0_{\lambda^*_{a,b}})\rightarrow0$.
		
		$(ii)$ Follows similarly from item $(i)$.
	\end{proof}
	Consider the sets
	\[\mathcal{N}^{-}_{\lambda^*_{a,b},d,c^-}=\{u\in \mathcal{N}^{-}_{\lambda^*_{a,b}}:\textrm{dist}(u,\mathcal{N}^{0}_{\lambda^*_{a,b}})\geq d,\|u\|_{X}\leq c^-\},\] and
	\[\mathcal{N}^{+}_{\lambda^*_{a,b},d,c^+}=\{v\in \mathcal{N}^{+}_{\lambda^*_{a,b}}:\textrm{dist}(v,\mathcal{N}^{0}_{\lambda^*_{a,b}})\geq d,\|v\|_{X}\geq c^+\},\]
	where $c^\pm \ \textrm{and} \ d$ are positive constants.

	In light of above proposition one can observe that    elements of $\mathcal{N}^{\pm}_{\lambda,d,c^\pm}$ can be projected over $\mathcal{N}^{\pm}_{\lambda}$ when $\lambda\downarrow\lambda^*_{a,b}$ i.e., if $u\in \mathcal{N}^\pm_{\lambda,d,c^\pm}$ there exists $\epsilon>0$ such that $u\in \hat{\mathcal{N}}_{\lambda}$ {or in $\hat{\mathcal{N}}_{\lambda} \cup \hat{\mathcal{N}}^{+}_{\lambda}$} for all $\lambda\in(\lambda^*_{a,b},\lambda^*_{a,b}+\epsilon)$.    
	Also it is important to observe that $\textrm{dist}(\mathcal{M}^{\pm}_{\lambda^*_{a,b}}, \mathcal{N}^{0}_{\lambda^*_{a,b}})>0$, where
	\[\mathcal{M}^{\pm}_{\lambda}=\{u\in \mathcal{N}^{\pm}_{\lambda}: \mathcal{J}^{\pm}_{\lambda}(u)=\hat{\mathcal{J}}^{\pm}_{\lambda}\}.\]
	In fact, $\textrm{dist}(\mathcal{M}^{\pm}_{\lambda^*_{a,b}}, \mathcal{N}^{0}_{\lambda^*_{a,b}})\to 0$ we can get a contradiction to the fact that no solution of $(P_{\lambda^*_{a,b}}) $ can be in $\mathcal{N}^{0}_{\lambda^*_{a,b}}$. 
	
	Clearly the sets $\mathcal{M}^{\pm}_{\lambda}\neq \emptyset$ for all $\lambda\in(0,\lambda^*_{a,b}]$. Now define 
	
	\[d^{+}_{\lambda^*_{a,b}}=\textrm{dist}(\mathcal{M}^{+}_{\lambda^*_{a,b}}, \mathcal{N}^0_{\lambda^*_{a,b}})  \  \ \textrm{and} \ \ d^{-}_{\lambda^*_{a,b}}=\textrm{dist}(\mathcal{M}^{-}_{\lambda^*_{a,b}}, \mathcal{N}^0_{\lambda^*_{a,b}}).\]
	
	Now choose $c^-_{\lambda^*_{a,b}}<c^-$ such that $\|u\|_{X}\leq c^-_{\lambda^*_{a,b}}$ for all $u\in \mathcal{M}^{-}_{\lambda^*_{a,b}}$ and $d^{-}\in (0,d^{-}_{\lambda^*_{a,b}})$.  With such controls for all $\lambda\in(\lambda^*_{a,b},\lambda^*_{a,b}+\epsilon)$ we will study following minimization problem
	\begin{equation}\label{eq 6.1}
		\hat{{\mathcal{J}}}^{-}_{\lambda,d^-,c^-}=\textrm{inf} \{\mathcal{J}^{-}_{\lambda}(u):u\in \mathcal{N}^{-}_{\lambda^*_{a,b},d^-,c^-}\}.
	\end{equation}
	In a similar way, we define with the choice of $d^+<d^{+}_{\lambda^*_{a,b}}$ and $c^+<c^+_{\lambda^*_{a,b}}$ where  $\|u\|_{X}\geq c^+_{\lambda^*_{a,b}}$ for all $u\in \mathcal{M}^{+}_{\lambda^*_{a,b}}$, 
	\begin{equation*}
		\hat{{\mathcal{J}}}^{+}_{\lambda,d^+,c^+}=\textrm{inf} \{\mathcal{J}^{+}_{\lambda}(v):v\in \mathcal{N}^{+}_{\lambda^*_{a,b},d^+,c^+}\}
	\end{equation*}
	for all $\lambda\in(\lambda^*_{a,b},\lambda^*_{a,b}+\epsilon)$. With such choice of $c^\pm,d^{\pm}$ we can observe that  $\mathcal{M}^{-}_{\lambda^*_{a,b}}\subset\mathcal{N}^{-}_{\lambda^*_{a,b},d^{-},c^-}$ as, if $u\in\mathcal{M}^{-}_{\lambda^*_{a,b}} $ we have $\textrm{dist}(u,\mathcal{N}^{0}_{\lambda^*_{a,b}})=d^{-}_{\lambda^*}> d^-$ implies $u\in \mathcal{N}^{-}_{\lambda^*_{a,b},d^{-},c^-}$. Similarly $\mathcal{M}^{+}_{\lambda^*_{a,b}}\subset\mathcal{N}^{+}_{\lambda^*_{a,b},d^{+},c^+}$.\\

	\begin{lemma}\label{prop6.1}
		For the above choices of $c^-,d^{-}$ there exits $\epsilon^{-}>0$ such that the functional $\mathcal{J}^{-}_{\lambda}$ constrained to $ \mathcal{N}^{-}_{\lambda^*_{a,b},d^-,c^-}$ has a minimizer ${u({\lambda})}\in$  $\mathcal{N}^{-}_{\lambda^*_{a,b},d^-,c^-}$ for all $\lambda\in(\lambda^*_{a,b},\lambda^*_{a,b}+\epsilon^-)$.
	\end{lemma}
	
	\begin{proof}
		For each $\lambda\in(\lambda^*_{a,b},\lambda^*_{a,b}+\epsilon)$, take a minimizing sequence $u_n(\lambda)\in \mathcal{N}^{-}_{\lambda^*_{a,b},d^{-},c^-}$ for $\hat{\mathcal{J}}^{-}_{\lambda,d^-,c^-}$. From Proposition \ref{prml}, we have $u_n(\lambda)\in \hat{\mathcal{N}}_{\lambda}$  and we can get $\delta<0$ such that $(t^{-}_{\lambda}(u_n(\lambda)))^2\psi_{\lambda,u_n}(t^{-}_{\lambda}(u_n(\lambda)))<\delta$ for all $\lambda\in(\lambda^*_{a,b},\lambda^*_{a,b}+\epsilon)$. Since $u_n(\lambda)$ is bounded, hence up to a subsequence $u_n(\lambda)\rightharpoonup u(\lambda)$ for some ${u(\lambda)}\not\equiv 0$ in $X$. 
		We target to show that, the weak limit ${u(\lambda)}$ belongs to $ \hat{\mathcal{N}}_{\lambda}$ for all $\lambda\in(\lambda^*_{a,b},\lambda^*_{a,b}+\epsilon^-)$.
		As $u_n(\lambda)\in \hat{\mathcal{N}}_{\lambda}$ for all $\lambda\in(\lambda^*_{a,b},\lambda^*_{a,b}+\epsilon^-)$, there exists unique scalar $t^{-}_{\lambda}(u_n(\lambda))>0$  for all $\lambda\in(\lambda^*_{a,b},\lambda^*_{a,b}+\epsilon^-)$. Moreover, since $t^{-}_{\lambda}(u_n(\lambda))<t^{-}_{\lambda^*_{a,b}}(u_n(\lambda))=1$   we get, $t^{-}_{\lambda}(u_n(\lambda))\rightarrow t({\lambda})\in (0,1)$. We claim that there exists $\epsilon^->0$ such that  $u(\lambda)\in \hat{\mathcal{N}}_{\lambda}$ for all $\lambda\in(\lambda^*_{a,b},\lambda^*_{a,b}+\epsilon^-)$. Suppose on contrary $u(\lambda)\notin \hat{\mathcal{N}}_{\lambda}$, thus there exists sequence $\lambda_k\downarrow \lambda^*_{a,b}$ such that $u(\lambda_k)\notin \hat{\mathcal{N}}_{\lambda_k}$ for large value of $k$. To get contradiction for this fact, assuming a minimizing  sequence $u_{n,k}\equiv t^{-}_{\lambda_k}(u_n(\lambda_k))u_n(\lambda_k)$ for $\hat{\mathcal{J}}^{-}_{\lambda_k,d^{-},c^-}$ for large $k\in \mathbb N$ (using $\eqref{eq 6.1}$) and  we show that it is also a minimizing sequence for $\hat{\mathcal{J}}^{-}_{\lambda^*_{a,b}}$. In the squeal we first prove following claim.
		
		\textit{Claim-} When $\lambda\downarrow\lambda^*_{a,b}$ function $\hat{\mathcal{J}}^{-}_{\lambda,d^{-},c^-}$ is decreasing and 
		\begin{equation}\label{claim}
			\displaystyle \lim_{\lambda\downarrow \lambda^*_{a,b}}\hat{\mathcal{J}}^{-}_{\lambda,d^{-},c^-}=\hat{\mathcal{J}}^{-}_{\lambda^*_{a,b}}.
		\end{equation}
		In order to prove this claim, using decreasing behaviour of $\mathcal{J}^{-}_{\lambda}(u)$ from  Proposition \eqref{p:2.10} $(ii)$ for all $u\in \hat{\mathcal{N}}^{-}_{\lambda^*_{a,b},d^{-},c^-}$, we have
		$\hat{\mathcal{J}}^{-}_{\lambda,d^{-},c^-}\leq\hat{\mathcal{J}}^{-}_{\lambda',d^{-},c^-}  $, where $\lambda^*_{a,b}<\lambda'<\lambda<\lambda^*_{a,b}+\epsilon$. Also when  $u_{\lambda^*_{a,b}}\in \mathcal{M}^-_{\lambda^*_{a,b}}$ we have $\hat{\mathcal{J}}^{-}_{\lambda,d^{-},c^-}\leq \mathcal{J}^{-}_{\lambda}(u_{\lambda^*_{a,b}})<\mathcal{J}^{-}_{\lambda^*_{a,b}}(u_{\lambda^*_{a,b}})=\hat{\mathcal{J}}^{-}_{\lambda^*_{a,b}}$ for all $\lambda\in (\lambda^*_{a,b},\lambda^*_{a,b}+\epsilon)$. 
		To prove \eqref{claim} suppose on contrary that there exists a sequence  $\lambda_{n}\downarrow \lambda^*_{a,b}$ or $\lambda_{n}\in (\lambda^*_{a,b},\lambda^*_{a,b}+\epsilon)$  for large $n$, such that
		\begin{equation*}
			\lim_{n\rightarrow\infty}\hat{\mathcal{J}}^{-}_{\lambda_n,d^{-},c^-}=\mathcal{J}<\hat{\mathcal{J}}^{-}_{\lambda^*_{a,b}}.
		\end{equation*}
		Also from equation \eqref{eq 6.1} we can have a sequence $\{u_k\}\subset \mathcal{N}^{-}_{\lambda^*_{a,b},d^-,c^-}$ such that for a given $\epsilon>0$ there exists $n_0,k_0\in \mathbb{N}$ such that 
		\begin{equation}\label{eq 6.4}
			|\mathcal{J}^{-}_{\lambda_n}(u_k)-\hat{{\mathcal{J}}}^-_{\lambda_{n},d^{-},c^-} |<\frac{\epsilon}{2},  \ \ \textrm{for all} \  n\geq n_0, \ k\geq k_0 . 
		\end{equation}
		Using continuity of $t^{-}_{\lambda}$ for $\lambda\in(\lambda^*_{a,b},\lambda^*_{a,b}+\epsilon)$ following Lemma \ref{p:2.10} $(i)$,
		we have
		\begin{equation}\label{eq 6.5}
			|\mathcal{J}^{-}_{\lambda_n}(u_k)-\mathcal{J}^{-}_{\lambda^*_{a,b}}(u_k) |<\frac{\epsilon}{2},  \ \ \textrm{for all} \  n\geq n_1 . 
		\end{equation}
		Taking  $n\geq n_2=\textrm{max}\{n_0,n_1\}$ and using \eqref{eq 6.4} and \eqref{eq 6.5}, we get
		\begin{align*}\hat{\mathcal{J}}^{-}_{\lambda^*_{a,b}}-\hat{\mathcal{J}}^{-}_{\lambda_n,d^-,c^-}&<\mathcal{J}^{-}_{\lambda^*_{a,b}}(u_k)-\hat{\mathcal{J}}^{-}_{\lambda_n,d^-,c^-}\\
			&\leq |\mathcal{J}^{-}_{\lambda^*_{a,b}}(u_k)-\mathcal{J}^{-}_{\lambda_n}(u_k)|+|\mathcal{J}^{-}_{\lambda_n}(u_k)-\hat{\mathcal{J}}^{-}_{\lambda_n,d^-,c^-}|\\
			&\leq \frac{\epsilon}{2}+\frac{\epsilon}{2}=\epsilon,
		\end{align*}
		implies, 
		\[\hat{\mathcal{J}}^{-}_{\lambda^*_{a,b}}< \hat{\mathcal{J}}^{-}_{\lambda_n,d^-,c^-}+\epsilon \ \textrm{for} \  \textrm{all} \ n\geq n_2.\] 
		Therefore, $\hat{\mathcal{J}}^{-}_{\lambda^*_{a,b}}\leq \mathcal{J}$ as $n\rightarrow \infty$  a contradiction. Thus $\lim_{\lambda\downarrow \lambda^*_{a,b}}\hat{\mathcal{J}}^{-}_{\lambda,d^{-},c^-}=\hat{\mathcal{J}}^{-}_{\lambda^*_{a,b}}$, which complete the proof of the claim.
		
		As a consequence of this claim, we have \begin{equation}\label{eq 5.61}
			|\hat{\mathcal{J}}^{-}_{\lambda^*_{a,b}}-\mathcal{J}^{-}_{\lambda_k}(u_{n,k})| \leq |\hat{\mathcal{J}}^{-}_{\lambda^*_{a,b}}-\hat{\mathcal{J}}^{-}_{\lambda_k,d^-,c^-}| +|\mathcal{J}^{-}_{\lambda_k}(u_{n,k})-\hat{\mathcal{J}}^{-}_{\lambda_k,d^-,c^-}|\rightarrow0
		\end{equation}
		
		as $n\rightarrow\infty$ followed by $k\to \infty$.  Therefore, up to a subsequence, $	u_{{n,k}} \rightharpoonup u$ in $X\setminus\{0\}$. 
		Now we claim that $u_{n,k}\rightarrow u$ in $X\setminus\{0\}$. Suppose on contrary,  by weak lower semi continuity of norm, we get
		\[\liminf_{n,k\rightarrow\infty}\psi'_{\lambda_{k},u_{n,k}}(t_{\lambda^*_{a,b}}(u))>\psi'_{\lambda^*_{a,b},u}(t_{\lambda^*_{a,b}}(u))=0.\]
		Hence $ t_{\lambda^*_{a,b}}(u)<t^{-}_{\lambda_{k}}(u_{n,k})$ for sufficiently large $n,k$. Therefore, from (\ref{eq 5.61}), we get
		\begin{align*}
			\mathcal E_{\lambda^*_{a,b}}(t_{\lambda^*_{a,b}}(u)u)&<\liminf_{n, k\to \infty}\mathcal E_{\lambda_k}(t_{\lambda^*_{a,b}}(u)u_{n, k}\\&<\liminf_{n, k\to \infty}\mathcal E_{\lambda_k}(t^{-}_{\lambda_k}(u_{n,k}))u_{n, k}=\hat{\mathcal{J}}^{-}_{\lambda^*_{a,b}}\end{align*}
		which is an absurd, therefore $u_{n,k}\rightarrow u$ in $X\setminus\{0\}$.
		Now,
		$u_{n,k}-u(\lambda) \rightharpoonup u(\lambda_k)-u(\lambda)$, implies
		\[\|u(\lambda_k)-u(\lambda)\|_{X}\leq \liminf_{n\rightarrow\infty} \|u_{n,k}-u(\lambda)\|_{X},\]
		for large $k$. Thus  $u(\lambda_k)\to u(\lambda)\in \mathcal{N}^{-}_{\lambda^*_{a,b},d^-,c^-}$ and consequently  $u(\lambda_k)\in \hat{\mathcal{N}}_{\lambda_k}$ for large $k$, which is a contradiction. Thus $u_n(\lambda)\rightarrow u(\lambda)$ in $X$ and $ u(\lambda)\in \hat{\mathcal{N}}_{\lambda}$ for all $\lambda\in(\lambda^*_{a,b},\lambda^*_{a,b}+\epsilon^-)$.
		{Subsequently, $t^{-}_{\lambda}(u_n(\lambda))u_n(\lambda)\rightarrow t(\lambda)u(\lambda)$  and $u(\lambda)\in \mathcal{N}^{-}_{\lambda^*_{a,b},d^-,c^-}$} with
		\[\mathcal{J}^{-}_{\lambda,d^-,c^-}=\mathcal{J}^{-}_{\lambda}(u(\lambda))\]
		that is, $u(\lambda)$ is minimizer for $\mathcal{J}^{-}_{\lambda,d^-,c^-}$ for all $\lambda\in(\lambda^*_{a,b},\lambda^*_{a,b}+\epsilon^-)$.
		
	\end{proof}

	\begin{lemma}
		For the above choices of $c^+,d^{+}$ there exits $\epsilon^{+}>0$ such that the functional $\mathcal{J}^{+}_{\lambda}$ constrained to $ \mathcal{N}^{+}_{\lambda^*_{a,b},d^+,c^+}$ has a minimizer ${v({\lambda})}\in$  $\mathcal{N}^{+}_{\lambda^*_{a,b},d^+,c^+}$ for all $\lambda\in(\lambda^*_{a,b},\lambda^*_{a,b}+\epsilon^+)$.
	\end{lemma}
	\begin{proof} The proof of the Lemma is similar to Lemma \ref{prop6.1} above.
	\end{proof}
	Now we choose  $\epsilon=\textrm{min}\{\epsilon^{-},\epsilon^+\}$ and rename the minimizers obtained above as $u_\lambda=u(\lambda)$ and $v_\lambda=v(\lambda)$ for all $\lambda\in (\lambda^*_{a,b}, \lambda^*_{a,b}+\epsilon)$.  For $\lambda\in (\lambda^*_{a,b},\lambda^*_{a,b}+\epsilon)$ the minimization problem posed in  $\mathcal{N}^{\pm}_{\lambda^*_{a,b},d^\pm,c^{\pm}}$ ensures that  minimizers $u_\lambda$ and $v_\lambda$ are  separated from $\mathcal{N}^{0}_{\lambda^*_{a,b}}$. Therefore, for $\epsilon>0$ sufficiently small chosen above,  $u_\lambda\in \mathcal{N}^-_{\lambda}$ and  $v_\lambda\in \mathcal{N}^+_{\lambda}$ for $\lambda\in (\lambda^*_{a,b},\lambda^*_{a,b}+\epsilon)$. Finally, invoking Theorem 2.3 of \cite{kj}, we get nontrivial weak solutions of \eqref{p} for $\lambda\in (\lambda^*_{a,b},\lambda^*_{a,b}+\epsilon)$. 
	
	\textbf{Proof of Theorem \ref{T1} and Theorem \ref{T2} :} The proof of Theorem  \ref{T1} and Theorem \ref{T2} follows by  combining the arguments of previous sections. Eventually, we get the nontrivial weak solutions of \eqref{p} for $\lambda\in (0,\lambda^*_{a,b}+\epsilon)$. Next, we show that the weak solutions are positive. In fact, due to the presence of nonlocal fractional operator, we have $\|u\|\neq \| \ |u| \ \|$ in $X$. As a consequence, $\mathcal{E}_{\lambda}(|u|)\neq \mathcal{E}_{\lambda}(u)$. To overcome this situation, we define the perturbed problem with nonlinearity as $\lambda f(x)(u^+)^{\gamma-1}+g(x)(u^+)^{p-1}$ and corresponding energy functional as follows
	\[\mathcal{E}^+_{\lambda}(u):=\frac{a}{2}\|u\|^2_{X}+\frac{b}{2\theta}\|u\|^{2\theta}_{X}-\frac{\lambda}{\gamma}\int_\Omega f(u^+)^{\gamma} dx -\frac{1}{p}\int_\Omega g(u^+)^p dx.\]
	It is easy to see that the critical points of $\mathcal{E}_{\lambda}$ are also the critical points of $\mathcal{E}^+_{\lambda}$. Consequently the weak solutions of the perturbed problem. Now testing the solutions of perturbed problem with test function $\phi=u^{-}$ and using
	\[(u(x)-u(y))(u^{-}(x)-u^{-}(y))\leq -|u^{-}(x)-u^{-}(y)|^2\]
	we get, $\|u^-\|^2_{X}=0$, thus $u$ is a non-negative solution of \eqref{p}. The positivity of the solutions follows via maximum principle (see \cite{silv}, Proposition 2.2.8) which completes the proof.
	
	\section*{\small Acknowledgments}
		The research of the first author is supported by Science and Engineering Research Board, Govt. of India,   grant SRG/2021/001076.	
	\section*{\small
		Conflict of interest} 
	
	{\small
		The authors declare that they have no conflict of interest.}
	
	
	
	
%

\end{document}